\documentclass[11pt]{amsart}
\usepackage[utf8]{inputenc}
\usepackage[T1]{fontenc}
\usepackage{amsmath,amssymb,amsthm}
\usepackage{mathtools}
\usepackage[margin=2.6cm]{geometry}
\usepackage{xcolor}
\usepackage{enumitem}
\usepackage[hidelinks,pdftitle={The Serrin exponent for Lane-Emden inequalities driven by sums of fractional p-Laplacians},pdfauthor={L. M. Del Pezzo, B. Pichucho, A. Quaas}]{hyperref}

\newtheorem{theorem}{Theorem}[section]
\newtheorem{proposition}[theorem]{Proposition}
\newtheorem{lemma}[theorem]{Lemma}
\newtheorem{corollary}[theorem]{Corollary}
\theoremstyle{definition}
\newtheorem{definition}[theorem]{Definition}
\newtheorem{example}[theorem]{Example}
\theoremstyle{remark}
\newtheorem{remark}[theorem]{Remark}
\numberwithin{equation}{section}

\newcommand{\R}{\mathbb{R}}
\newcommand{\RN}{\mathbb{R}^N}
\newcommand{\Lc}{\mathcal{L}}
\newcommand{\DL}{\mathcal{D}_{\mathcal L}}
\newcommand{\E}{\mathcal{E}}
\newcommand{\kb}{\overline{\kappa}}
\newcommand{\ku}{\underline{\kappa}}
\newcommand{\qz}{\mathfrak{q}_0}
\newcommand{\qL}{\mathfrak{q}_{\mathcal L}}

\newcommand{\cc}{\mathfrak{c}}
\newcommand{\loc}{\mathrm{loc}}
\newcommand{\dist}{\operatorname{dist}}
\newcommand{\Nm}{I_-}

\begin{document}

\title[Serrin exponent for sums of fractional \(p\)-Laplacians]{The Serrin exponent for Lane--Emden inequalities driven by sums of fractional \(p\)-Laplacians}

\author{Leandro M. Del Pezzo}
\address{IESTA -- Departamento de M\'etodos Cuantitativos, Facultad de Ciencias Econ\'omicas y de Administraci\'on, Universidad de la Rep\'ublica, Av.~Gonzalo Ram\'irez 1926, 11200 Montevideo, Uruguay}
\email{leandro.delpezzo@fcea.edu.uy}
\author{Bryan Pichucho}
\address{Departamento de Matem\'atica, Universidad T\'ecnica Federico Santa Mar\'ia, Avenida Espa\~na 1680, Valpara\'iso, Chile}
\email{bpichucho@usm.cl}
\author{Alexander Quaas}
\address{Departamento de Matem\'atica, Universidad T\'ecnica Federico Santa Mar\'ia, Avenida Espa\~na 1680, Valpara\'iso, Chile}
\email{alexander.quaas@usm.cl}

\date{\today}

\begin{abstract}
	We study positive weak super-solutions of the Lane--Emden inequality
	\[
		\Lc u\coloneqq\sum_{i=1}^n(-\Delta_{p_i})^{s_i}u\ge u^q\quad\text{in }\RN,
	\]
	where each \((-\Delta_{p_i})^{s_i}\) is a fractional \(p_i\)-Laplacian of order \(s_i\in(0,1]\) and exponent \(p_i\in(1,\infty)\), and \(\max_is_ip_i<N\).
	Let \(\kappa_i\coloneqq(N-s_ip_i)/(p_i-1)\) be the decay exponent of the fundamental solution of the \(i\)-th operator, and let
	\(\kb\coloneqq\max_i\kappa_i\). We show that the Serrin exponent of \(\Lc\) is \(\qL=N/\kb\). More precisely, for \(0<q<\qL\) the only
	non-negative continuous weak super-solution is \(u\equiv0\), while for every \(q>\qL\) we construct explicit positive weak super-solutions
	of the form \(A(\lambda+|x|)^{-\kappa}\). In other words, the Serrin exponent of the sum is the smallest of the Serrin exponents \(N/\kappa_i\)
	of the single operators, that is, the one of the operator whose fundamental solution decays fastest. No structural assumption relating the
	different operators is needed. The endpoint \(q=\qL\) is left open. The local order \(s_i=1\) is treated together with the nonlocal orders,
	and all the arguments are carried out in the weak setting.
\end{abstract}
\maketitle

%=====================================================================
\section{Introduction}

	A Liouville-type theorem asserts that an equation or inequality in the whole of \(\RN\) has no non-trivial solution in a given class. Combined with the 
	rescaling method of Gidas and Spruck, such results yield a priori bounds, universal estimates and decay rates, see for instance \cite{GS81a,GS81b,SZ02}.

	We first recall the local case. For \(p\in(1,\infty)\), the \(p\)-Laplacian is
	\begin{equation}\label{eq:pLap}
		-\Delta_pu=-\operatorname{div}\left(|\nabla u|^{p-2}\nabla u\right),
	\end{equation}
	and the model problem is the Lane--Emden equation
	\[
		-\Delta_pu=u^q\quad\text{in }\RN.
	\]
	Two different critical exponents are attached to this problem. For solutions and \(p=2\), the classical result of Gidas and Spruck \cite{GS81a,GS81b} gives 
	nonexistence of positive solutions exactly below the Sobolev exponent \(\frac{N+2}{N-2}\). For super-solutions the relevant exponent is smaller. When 
	\(1<p<N\), there is no positive super-solution of 
	\[
		-\Delta_pu\ge u^q\quad\text{in }\RN
	\]
	when \(0<q\le\frac{N(p-1)}{N-p}\), while positive super-solutions exist above this exponent. See Gidas \cite{Gid80} for \(p=2\), Mitidieri and Pohozaev 
	\cite{MP99} and, through a viscosity--comparison argument, Birindelli and Demengel \cite{BD02} (see also Serrin and Zou \cite{SZ02}). The number
	\(\frac{N(p-1)}{N-p}\) is known as the \emph{Serrin exponent} of \(-\Delta_p\), after the work of Serrin on isolated singularities of
	quasilinear equations \cite{Ser64,Ser65}. For \(p=2\) the value \(\frac N{N-2}\) is sometimes also called the \emph{Gidas exponent}, after
	\cite{Gid80}. In this paper we deal with super-solutions, and hence with the Serrin exponent and its nonlocal counterparts.

	\medskip

	We now turn to the nonlocal case. For \(s\in(0,1)\), the fractional \(p\)-Laplacian is
	\begin{equation}\label{eq:fracpLap}
		(-\Delta_p)^su(x)\coloneqq2\,C_{N,s,p}\lim_{\varepsilon\to0^+}\int_{\RN\setminus 
			B_\varepsilon(x)}\frac{\Psi_p(u(x)-u(y))}{|x-y|^{N+sp}}\,dy,\qquad\Psi_p(t)\coloneqq|t|^{p-2}t,
	\end{equation}
	where \(C_{N,s,p}>0\) is a normalization constant. 
	For smooth \(u\), one has \(\lim_{s\nearrow1}(-\Delta_p)^su=-\Delta_pu\) (at points where \(\nabla u\ne0\), 
	if \(p<2\)), see \cite{BS22,DNPV12}. In this sense \(-\Delta_p\) is the endpoint 
	\(s=1\) of the scale, so we write \((-\Delta_p)^1=-\Delta_p\) and treat the local order together with the fractional ones. For simplicity, 
	from now on we omit the constant \(C_{N,s,p}\), since it does not affect the qualitative properties of the solutions we study. A probabilistic 
	interpretation of the fractional \(p\)-Laplacian is given in \cite{EJN26}.

	For nonlinear nonlocal operators, the first Liouville theorems and fundamental solutions are due to Felmer and Quaas \cite{FQ11}. The corresponding Lane--Emden equation is
	\begin{equation}\label{eq:LE-single}
		(-\Delta_p)^su=u^q\quad\text{in }\RN,
	\end{equation}
	which was studied by Del Pezzo and Quaas \cite{DPQ26}, who computed the fundamental solution of \((-\Delta_p)^s\) and proved a dichotomy for \eqref{eq:LE-single} at 
	the exponent \(\frac{N(p-1)}{N-sp}\), with no positive super-solution below it and a positive super-solution above it, namely 
	\(A(1+|x|)^{-\kappa}\) with \(\kappa=\frac{sp}{q-p+1}\) for a suitable \(A>0\).   This exponent is the Serrin exponent of \((-\Delta_p)^s\), and for \(s=1\) it reduces to the Serrin exponent of \(-\Delta_p\) recalled above. We also refer to 
	\cite{Liu25}, where nonexistence is obtained for weak super-solutions in \(W^{s,p}_\loc\) and the critical exponent is also addressed (see Remark~\ref{rem:endpoint}). Kim and Lee \cite{KL26} proved a Liouville theorem for 
	\(\mathcal L_Ku\ge|u|^q\), where \(\mathcal L_K\) is a linear operator with an even kernel such that \(0\le K(z)\le\Lambda|z|^{-N-2s}\), in the range 
	\(1\le q\le N/(N-2s)\).

	\medskip

	Sums of operators have been considered only recently. For the mixed operator 
	\(-\Delta+(-\Delta)^s\), Li, Wang and Xie \cite{LWX26} obtained exact asymptotic expansions of the fundamental solution near the origin and at infinity, 
	and Guo, Li and Xie \cite{GLX26} proved that, for \(N\ge3\) and \(q>1\), the inequality
	\[
		-\Delta u+(-\Delta)^su\ge u^q\quad\text{in }\RN
	\]
	has positive super-solutions if and only if \(q>N/(N-2s)\), which is the Serrin exponent of \((-\Delta)^s\) alone. For the local \((p_1,p_2)\)-Laplacian, 
	Bhakta, Biswas and Filippucci \cite{BBF26} showed that the Serrin exponent of the sum is that of the operator with the smaller homogeneity, that is, with the larger decay exponent \((N-p)/(p-1)\). In both cases the Serrin exponent of the sum is the one of the operator whose fundamental solution decays faster. Barrios, Del Pezzo and Quaas \cite{BDPQ25} studied mixed local--nonlocal 
	Laplacians without a standard critical exponent. Sums of this kind also appear in the double-phase and \((p,q)\)-growth problems going back to 
	Marcellini \cite{Mar91}, in mixed local--nonlocal problems \cite{BDVV22,GK22}, and in superpositions of fractional operators \cite{DPSV25,DPLSV25}.

	\medskip

	\noindent\textbf{Aim of this paper.} We identify the Serrin exponent of a finite sum of fractional \(p\)-Laplacians whose orders and 
	homogeneities may differ, and we establish the corresponding dichotomy. This covers in a single framework sums of local operators, sums of nonlocal operators and mixed sums.

	Fix \(n\in\mathbb N\), orders \(s_1,\dots,s_n\in(0,1]\) and exponents 
	\(p_1,\dots,p_n\in(1,\infty)\), and set
	\begin{equation}\label{eq:L}
		\Lc u(x)\coloneqq\sum_{i=1}^n(-\Delta_{p_i})^{s_i}u(x).
	\end{equation}
	We study the Lane--Emden problem
	\begin{equation}\label{eq:LE}
		\Lc u=u^q\quad\text{in }\RN,
	\end{equation}
	which reduces to \eqref{eq:LE-single} when \(n=1\) and brings in the local \(p\)-Laplacian whenever some \(s_i=1\). All our results hold verbatim for 
	\(\sum_ia_i(-\Delta_{p_i})^{s_i}\) with constants \(a_i>0\) (see Remark~\ref{rem:weights}).

	The operators in \(\Lc\) scale differently. To make this precise, we write
	\begin{equation}\label{eq:kappai}
		\kappa_i\coloneqq\frac{N-s_ip_i}{p_i-1},\qquad i\in\{1,\dots,n\},
	\end{equation}
	for the \emph{decay exponent} of the \(i\)-th operator, since the power \(|x|^{-\kappa_i}\) is, up to a multiplicative constant, the fundamental solution of \((-\Delta_{p_i})^{s_i}\). When the radial profile \(w_\kappa(x)=(1+|x|)^{-\kappa}\) is inserted in \(\Lc\), the \(i\)-th term is bounded from below by
	\[
		\cc_{s_i,p_i}(-\kappa)\,(1+|x|)^{-(\kappa(p_i-1)+s_ip_i)},
	\]
	provided \(\kappa<N/(p_i-1)\) when \(s_i<1\) (for larger \(\kappa\) see Proposition~\ref{prop:large}), where the constant \(\cc_{s_i,p_i}(-\kappa)\) is positive for \(\kappa<\kappa_i\), vanishes at \(\kappa=\kappa_i\) and is negative for \(\kappa>\kappa_i\) (Lemma~\ref{lem:fundamental} and Proposition~\ref{prop:pointwise}), while \(w_\kappa^q=(1+|x|)^{-\kappa q}\). If one requires all these constants to be non-negative, one is forced to take \(\kappa\le\min_i\kappa_i\), and, unless all the decay exponents coincide, the resulting range of exponents \(q\) is not optimal (Remark~\ref{rem:oneparam}). The point is that the terms with a negative constant are harmless. Indeed, since \((p_i-1)\kappa_i=N-s_ip_i\), we have
	\begin{equation}\label{eq:key}
		\kappa(p_i-1)+s_ip_i=N+(p_i-1)(\kappa-\kappa_i),
	\end{equation}
	so a term with \(\kappa_i<\kappa\) decays faster than \(|x|^{-N}\), while a term with \(\kappa_i>\kappa\) decays more slowly than \(|x|^{-N}\). After the rescaling \(W(x)=A(\lambda+|x|)^{-\kappa}\) with \(\lambda\) large, one term with a positive constant absorbs all the negative ones. This allows us to take \(\kappa\) up to the \emph{largest} decay exponent.

	We set
	\begin{equation}\label{eq:Qi}
		Q_i(\kappa)\coloneqq p_i-1+\frac{s_ip_i}\kappa\quad(\kappa>0),\qquad\kb\coloneqq\max_{1\le i\le n}\kappa_i.
	\end{equation}
		By \eqref{eq:key}, the Serrin exponent of the \(i\)-th operator alone is \(Q_i(\kappa_i)=N/\kappa_i\). The Serrin exponent of the sum will be
	\begin{equation}\label{eq:qL}
		\qL\coloneqq\frac N\kb=\min_{1\le i\le n}Q_i(\kb),
	\end{equation}
	where the second equality is proved in Lemma~\ref{lem:thresholds}. Both the existence of super-solutions for \(q>\qL\) and the nonexistence for \(0<q<\qL\) hold without further assumptions. The difficulties come from the fractional indices \(i\) with \(\kb\ge N/(p_i-1)\), and we explain after the statement of the main result how we deal with them.

	\medskip
	\noindent\textbf{Main result.} Our main result is the following dichotomy.

	\begin{theorem}[Sharp dichotomy]\label{thm:main}
		Let \(N\ge2\), \(n\ge1\), \(p_i\in(1,\infty)\) and \(s_i\in(0,1]\) for \(i\in\{1,\dots,n\}\), with \(\max_is_ip_i<N\). Then the following hold.
		\begin{enumerate}[label=\textup{(\roman*)}]
			\item If \(0<q<\qL\), the only non-negative weak super-solution \(u\in\DL(\RN)\cap C(\RN)\) of \eqref{eq:LE} (in the sense of Definition~\ref{def:weak}) is \(u\equiv0\).
			\item If \(q>\qL\), problem \eqref{eq:LE} admits positive weak super-solutions of the form \(W(x)=A(\lambda+|x|)^{-\kappa}\).
		\end{enumerate}
	\end{theorem}

	Here \(\DL(\RN)\) is the natural domain of \(\Lc\), defined in \eqref{eq:DL}. We prove Theorem~\ref{thm:main} in two parts. Part (ii) is Theorem~\ref{thm:existence} in Section~\ref{sec:existence}, which also gives the admissible ranges of \(\kappa\), \(A\) and \(\lambda\). Part (i) is Theorem~\ref{thm:liouville} in Section~\ref{sec:liouville}.

	\medskip

	To be more concrete, we write down Theorem~\ref{thm:main} for two operators. This is only an example, but it already shows how the threshold is selected.

	\begin{corollary}[Two operators, \(n=2\)]\label{cor:n2}
		Let \(N\ge2\) and, for \(i=1,2\), let \(p_i\in(1,\infty)\) and \(s_i\in(0,1]\) with \(\max_is_ip_i<N\). Order the operators so that \(\kappa_1\le\kappa_2\), and consider
		\begin{equation}\label{eq:LE2}
			(-\Delta_{p_1})^{s_1}u+(-\Delta_{p_2})^{s_2}u=u^q\quad\text{in }\RN.
		\end{equation}
		Then the following hold.
		\begin{enumerate}[label=\textup{(\roman*)}]
			\item If \(0<q<N/\kappa_2\), the only non-negative weak super-solution \(u\in\DL(\RN)\cap C(\RN)\) of \eqref{eq:LE2} is \(u\equiv0\).
			\item If \(q>N/\kappa_2\), problem \eqref{eq:LE2} admits positive weak super-solutions of the form \(W(x)=A(\lambda+|x|)^{-\kappa}\).
		\end{enumerate}
	\end{corollary}

	Since \(\kb=\kappa_2\), this is Theorem~\ref{thm:main} with \(n=2\). The threshold \(N/\kappa_2\) is set by the operator with the \emph{larger} decay exponent, that is, by the operator whose fundamental solution decays faster. In part (ii) one takes \(\kappa<\kappa_2\) close to \(\kappa_2\). When \(\kappa_1<\kappa<\kappa_2\), the first operator contributes a negative term, which is absorbed by the second one. By \eqref{eq:key}, this negative term decays faster than \(|x|^{-N}\), while the positive one decays more slowly, and the absorption is achieved through the parameter \(\lambda\). For instance, for \(-\Delta+(-\Delta)^{1/2}\) in \(\R^3\) we have \(\kappa_1=1\) and \(\kappa_2=2\), so the threshold is \(3/2\) and not the exponent \(3\) of the Laplacian alone (see Example~\ref{ex:three}).

	\medskip

	To reach the threshold \(N/\kb\) in part (ii), we need barriers \(W\) with \(\kappa\) close to \(\kb\). If \(\kb>N/(p_i-1)\) for some fractional index \(i\), then \(\kappa(p_i-1)\ge N\) for \(\kappa\) close to \(\kb\). Then \(|x|^{-\kappa(p_i-1)}\) is not integrable at the origin, and the comparison with the fundamental solution that gives the lower bound of Proposition~\ref{prop:pointwise}, as in \cite{DPQ26}, is no longer available. Proposition~\ref{prop:large} fills this gap. It gives
	\[
		(-\Delta_{p_i})^{s_i}w_\kappa(x)\ge-C(1+|x|)^{-(N+s_ip_i)}\log(e+|x|)\qquad\text{for }x\ne0,
	\]
	and the proof shows where the negative part comes from: it is the mass of \(w_\kappa^{p_i-1}\) near the origin, seen from far away. This term decays faster than \(|x|^{-N}\), so it is absorbed like the other negative terms. Without this estimate the barriers would be limited to \(\kappa<N/(p_i-1)\) for every fractional index \(i\), and in general the threshold \(N/\kb\) would be out of reach.

	\medskip

	The proof of part (i) follows the scheme of \cite{DPQ26}, and its new ingredient is Lemma~\ref{lem:trunc}. A super-solution is bounded from below by comparison with truncated powers \(|x|^\beta\), and the argument needs \(\beta<-\kb\) arbitrarily close to \(-\kb\). If \(\kb\ge N/(p_i-1)\) for some fractional index \(i\), we are forced to take \(\beta\le-N/(p_i-1)\). In that range \((-\Delta_{p_i})^{s_i}|x|^\beta\) is not even defined, and the estimates of \cite[Lemmas~4.1 and~4.2]{DPQ26} do not apply. Lemma~\ref{lem:trunc} shows that the truncated powers \(\min\{\varepsilon^\beta,|x|^\beta\}\) are still sub-solutions far from the origin, for every \(\beta<-\kappa_i\). In the new range the reason is the one behind Proposition~\ref{prop:large}. Flattening the power puts a large mass near the origin, and this mass produces a negative term of order \(|x|^{-N-s_ip_i}\), which wins at large distances. As a consequence, part (i) needs no condition relating the different operators. This is a phenomenon of sums: for a single operator \(\kappa_1<N/(p_1-1)\), and the new range never occurs. Example~\ref{ex:three}(b) shows a sum where it does.

	\medskip

	For \(n=1\), Theorem~\ref{thm:main} is contained in \cite[Theorems~6.1 and~6.2]{DPQ26}, and the nonexistence part for \(q<\qL\) was extended by Liu \cite{Liu25} to weak super-solutions in \(W^{s,p}_\loc\), without continuity. For \(-\Delta+(-\Delta)^s\) the threshold \(N/(N-2s)\) agrees with \cite{GLX26}, and for the sum of two local operators with the exponent found in \cite{BBF26}. We stress that in the model case \(-\Delta+(-\Delta)^s\) the result of \cite{GLX26} is stronger than ours: it applies to distributional super-solutions and includes the endpoint \(q=N/(N-2s)\). On the other hand, it requires \(q>1\), while part (i) of Theorem~\ref{thm:main} also covers the sublinear range \(0<q\le1\). The operator may combine different orders \(s_i\) and different homogeneities \(p_i\), a situation not covered by \cite{GLX26,BBF26,KL26}. The proofs also differ: the negative terms are absorbed through the parameter \(\lambda\), and all the arguments are carried out for weak super-solutions.

	We do not treat the endpoint \(q=N/\kb\). It belongs to the nonexistence range for a single local operator \cite{MP99,SZ02} and for \(-\Delta+(-\Delta)^s\) \cite{GLX26}. For a single fractional \(p\)-Laplacian it is, to the best of our knowledge, still open, and for general sums our argument stops exactly there (see Remark~\ref{rem:endpoint}).

	A further difference with \cite{DPQ26} is the framework. There, the contact argument at the end of the proof of the Liouville theorem is carried out for viscosity solutions, which for a single operator is justified by the equivalence between weak and viscosity solutions \cite{BM21,KKL19}. For the sum \(\Lc\), and in the singular range \(1<p_i<2\), this equivalence is delicate (see \cite{LG26} for \(-\Delta_p+(-\Delta_p)^s\)). We avoid it altogether. The strong minimum principle and the contact estimate are proved directly for weak super-solutions, by means of the comparison principle and of a cut-off function which is flat to infinite order where its gradient vanishes.

	\medskip
	\noindent\textbf{Plan of the paper.} In Section~\ref{sec:prelim}, we fix notation, introduce weak solutions of \(\Lc\), show that pointwise and weak inequalities agree for the explicit barriers used later, and prove the comparison and strong minimum principles. In Section~\ref{sec:barrier}, we establish the pointwise lower bound for the radial barrier \(w_\kappa(x)=(1+|x|)^{-\kappa}\), for every sign of the fundamental constant, together with a cruder lower bound for large \(\kappa\). In Section~\ref{sec:existence}, we prove part (ii) of Theorem~\ref{thm:main}, in the more precise form of Theorem~\ref{thm:existence}. In Section~\ref{sec:liouville}, we prove the estimates for truncated powers, the Hadamard estimates and the contact estimate, and from them we deduce part (i) of Theorem~\ref{thm:main}, stated there as Theorem~\ref{thm:liouville}. We also comment on the endpoint. Finally, in Section~\ref{sec:examples}, we give examples.
%=====================================================================
\section{Preliminaries and notation}\label{sec:prelim}

	This section sets up what we need in the rest of the paper. We begin with the function spaces suited to the multi-operator setting, the forms associated with each operator and with their sum, and the notion of weak sub- and super-solution (Section~\ref{subsec:spaces}), and then we record some elementary properties of these forms (Section~\ref{subsec:forms}). Every solution in this paper is taken in this weak sense. In Section~\ref{subsec:pointwise} we show that, for the explicit functions used as barriers, pointwise inequalities can be integrated into weak ones. We close the section with the comparison and strong minimum principles for \(\Lc\) (Section~\ref{subsec:comparison}).

\subsection{Function spaces, forms and weak solutions}\label{subsec:spaces}
	Throughout, \(\Omega\) is an open subset of \(\RN\), and \(s\in(0,1)\), \(p\in(1,\infty)\). We denote by \(W^{s,p}(\Omega)\) the fractional Sobolev space
	\[
		W^{s,p}(\Omega)\coloneqq\left\{u\in L^p(\Omega):\ [u]_{W^{s,p}(\Omega)}<\infty\right\},
	\]
	with the Gagliardo seminorm
	\[
		[u]_{W^{s,p}(\Omega)}\coloneqq\left(\iint_{\Omega\times\Omega}\frac{|u(x)-u(y)|^p}{|x-y|^{N+sp}}\,dx\,dy\right)^{1/p}.
	\]
	The space is normed by \(\|u\|_{W^{s,p}(\Omega)}\coloneqq\left(\|u\|^p_{L^p(\Omega)}+[u]^p_{W^{s,p}(\Omega)}\right)^{1/p}\). In the local case \(s=1\), \(W^{1,p}(\Omega)\) is the classical Sobolev space. We write
	\[
		W^{s,p}_\loc(\Omega)\coloneqq\left\{u\in L^p_\loc(\Omega):\ u\in W^{s,p}(\Omega')\text{ for every open set }\Omega'\Subset\Omega\right\},
	\]
	and analogously \(W^{1,p}_\loc(\Omega)\) in the local case. For \(s\in(0,1)\) we also use the tail space
	\begin{equation}\label{eq:tail}
		L^{p-1}_{sp}(\RN)\coloneqq\left\{u\in L^{p-1}_\loc(\RN):\ \int_{\RN}\frac{|u(x)|^{p-1}}{(1+|x|)^{N+sp}}\,dx<\infty\right\}.
	\end{equation}
	We split the indices into fractional and local ones,
	\begin{equation}\label{eq:FL}
		F\coloneqq\{i:\ s_i\in(0,1)\},\qquad L\coloneqq\{i:\ s_i=1\},\qquad F\cup L=\{1,\dots,n\}.
	\end{equation}
	The natural domain of \(\Lc\) is
	\begin{equation}\label{eq:DL}
		\DL(\Omega)\coloneqq\bigcap_{i\in F}\left(W^{s_i,p_i}_\loc(\Omega)\cap L^{p_i-1}_{s_ip_i}(\RN)\right)\ \cap\ \bigcap_{i\in L}W^{1,p_i}_\loc(\Omega).
	\end{equation}
	Observe that every bounded and locally Lipschitz function \(u:\RN\to\R\) belongs to \(\DL(\RN)\). We also use the notation \(\Psi_p\) of \eqref{eq:fracpLap} for vectors, \(\Psi_p(\xi)\coloneqq|\xi|^{p-2}\xi\) for \(\xi\in\RN\).

	To each \((-\Delta_{p_i})^{s_i}\) we attach a form \(\E_{p_i,s_i}(u,\varphi)\), linear in the second argument. It is, formally, the first variation of the energy
	\[
		\frac1{p_i}[u]^{p_i}_{W^{s_i,p_i}(\RN)}\quad\text{if }s_i\in(0,1),\qquad\frac1{p_i}\|\nabla u\|^{p_i}_{L^{p_i}}\quad\text{if }s_i=1.
	\]
	Let \(u\in\DL(\Omega)\) and \(\varphi\in C^\infty_c(\Omega)\). In the nonlocal case \(s_i\in(0,1)\) we set
	\begin{equation}\label{eq:Enl}
		\E_{p_i,s_i}(u,\varphi)\coloneqq\iint_{\RN\times\RN}\frac{\Psi_{p_i}(u(x)-u(y))(\varphi(x)-\varphi(y))}{|x-y|^{N+s_ip_i}}\,dx\,dy,
	\end{equation}
	and in the local case \(s_i=1\) we set
	\begin{equation}\label{eq:Eloc}
		\E_{p_i,1}(u,\varphi)\coloneqq\int_\Omega\Psi_{p_i}(\nabla u)\cdot\nabla\varphi\,dx.
	\end{equation}
	Both are absolutely convergent integrals (see Remark~\ref{rem:welldef} below). We write
	\[
		\E_\Lc(u,\varphi)\coloneqq\sum_{i=1}^n\E_{p_i,s_i}(u,\varphi).
	\]

	\begin{definition}[Weak super- and sub-solutions]\label{def:weak}
		Let \(f:\Omega\times\R\to\R\). A function \(u:\RN\to\R\) is a \emph{weak super-solution} (\emph{sub-solution}) of
		\begin{equation}\label{eq:Lf}
			\Lc u=f(x,u)\quad\text{in }\Omega,
		\end{equation}
		if
		\begin{enumerate}[label=\textup{(\arabic*)}]
			\item \(u\in\DL(\Omega)\),
			\item \(f(\cdot,u)\in L^1_\loc(\Omega)\),
			\item for every non-negative \(\varphi\in C^\infty_c(\Omega)\),
			\[
				\E_\Lc(u,\varphi)\ \ge(\le)\ \int_\Omega f(x,u)\,\varphi\,dx.
			\]
		\end{enumerate}
		Moreover, \(u\) is a \emph{weak solution} of \eqref{eq:Lf} if it is both a weak super- and sub-solution. We also say that \(\Lc u\ge f\) (resp.\ \(\Lc u\le f\)) \emph{weakly in} \(\Omega\) when \(u\) is a weak super-solution (respectively sub-solution) of \eqref{eq:Lf}.
	\end{definition}

	\begin{remark}\label{rem:welldef}
		The definition is well posed; see, for instance, \cite{DKP16,KKL19} for a single operator. Fix \(\varphi\in C^\infty_c(\Omega)\), let \(K\coloneqq\operatorname{supp}\varphi\), and choose an open set \(U\) with \(K\subset U\Subset\Omega\). For \(i\in L\), H\"older's inequality gives \(|\E_{p_i,1}(u,\varphi)|\le\|\nabla u\|^{p_i-1}_{L^{p_i}(U)}\|\nabla\varphi\|_{L^{p_i}(U)}\). For \(i\in F\), we split \(\RN\times\RN\) into \(U\times U\) and its complement. On \(U\times U\) we use H\"older's inequality, and on the rest we use that \(\varphi\) vanishes outside \(K\) and that \(|x-y|\ge\dist(K,\partial U)\) and \(1+|y|\le C|x-y|\) for \(x\in K\) and \(y\notin U\). This gives
		\begin{equation}\label{eq:formbound}
			\iint_{\RN\times\RN}\frac{|u(x)-u(y)|^{p_i-1}|\varphi(x)-\varphi(y)|}{|x-y|^{N+s_ip_i}}\,dx\,dy\le[u]^{p_i-1}_{W^{s_i,p_i}(U)}[\varphi]_{W^{s_i,p_i}(U)}+C_u\|\varphi\|_{L^{p_i}(U)},
		\end{equation}
		where \(C_u\) depends on \(K\), \(U\), \(N\), \(s_i\), \(p_i\), \(\|u\|_{L^{p_i}(U)}\) and on the tail \(\int_{\RN}|u|^{p_i-1}(1+|y|)^{-N-s_ip_i}\,dy\). Finally, the right-hand side of Definition~\ref{def:weak}(3) is finite by (2).
	\end{remark}

\subsection{Properties of the forms}\label{subsec:forms}
	For \(u,v\in\DL(\Omega)\) we say that \(\Lc u\ge\Lc v\) \emph{weakly in} \(\Omega\) if
	\[
		\E_\Lc(u,\varphi)\ge\E_\Lc(v,\varphi)\qquad\text{for every non-negative }\varphi\in C^\infty_c(\Omega).
	\]
	By \eqref{eq:formbound} and its local analogue, the same formulas define \(\E_{p_i,s_i}(u,\varphi)\) when \(\varphi\) has compact support in \(\Omega\) and belongs to \(W^{s_i,p_i}(\RN)\) if \(i\in F\), or to \(W^{1,p_i}(\RN)\) if \(i\in L\). Moreover, for functions \(\varphi\) supported in a fixed compact subset of \(\Omega\), the map \(\varphi\mapsto\E_{p_i,s_i}(u,\varphi)\) is continuous with respect to the norm of that space.

	Each \(\E_{p_i,s_i}\) is linear with respect to the test function, and so is \(\E_\Lc\). This is what allows us to reduce statements about the sum to statements involving a single operator. We record two elementary invariances.

	\begin{remark}[Affine invariance and scaling]\label{rem:affine}
		Adding a constant changes neither \(v(x)-v(y)\) nor \(\nabla v\), and \(\Psi_{p_i}\) is \((p_i-1)\)-homogeneous. Hence, for \(\alpha>0\), \(c\in\R\), \(v\in\DL(\Omega)\) and \(\varphi\in C^\infty_c(\Omega)\),
		\begin{equation}\label{eq:affine}
			\E_\Lc(\alpha v+c,\varphi)=\sum_{i=1}^n\alpha^{p_i-1}\E_{p_i,s_i}(v,\varphi).
		\end{equation}
		So if \(\E_{p_i,s_i}(v,\varphi)\le0\) for every \(i\), then \(\E_\Lc(\alpha v+c,\varphi)\le0\). The sign of each summand is needed here, because the factors \(\alpha^{p_i-1}\) are different. In the same way, the change of variables \(y\mapsto z+\lambda y\) gives, for \(\lambda>0\) and \(z\in\RN\),
		\begin{equation}\label{eq:scaling}
			(-\Delta_{p_i})^{s_i}\left(\alpha\,v(\tfrac{\cdot-z}\lambda)+c\right)(x)=\alpha^{p_i-1}\lambda^{-s_ip_i}\,\left[(-\Delta_{p_i})^{s_i}v\right](\tfrac{x-z}\lambda),
		\end{equation}
		whenever the right-hand side exists.
	\end{remark}

	\begin{remark}[Weights]\label{rem:weights}
		Since the normalization constants have been dropped, it is worth noting that all the results of this paper hold, with the same proofs, for the operator \(\sum_{i=1}^na_i(-\Delta_{p_i})^{s_i}\) with constants \(a_i>0\). The only change is that each constant \(\cc_{s_i,p_i}\) below is multiplied by \(a_i\), which does not affect its sign.
	\end{remark}

\subsection{From pointwise to weak inequalities}\label{subsec:pointwise}
	As explained in the introduction, we work only with weak super-solutions and never evaluate \(\Lc u\) at a point. The operator is evaluated pointwise only at explicit functions, namely the truncated powers of Section~\ref{sec:liouville}, the barrier \(W(x)=A(\lambda+|x|)^{-\kappa}\) and the flat cut-off functions. Lemmas~\ref{lem:local} to~\ref{lem:flat} show that, for these functions, the principal value exists and pointwise inequalities can be integrated against non-negative test functions.

	\medskip

	For \(s\in(0,1)\), \(x\in\RN\) and \(\varepsilon>0\) we consider the truncated operator
	\[
		(-\Delta_p)^s_\varepsilon u(x)\coloneqq2\int_{\RN\setminus B_\varepsilon(x)}\frac{\Psi_p(u(x)-u(y))}{|x-y|^{N+sp}}\,dy,
	\]
	so that \((-\Delta_p)^s_\varepsilon u(x)\to(-\Delta_p)^su(x)\) as \(\varepsilon\to0^+\), whenever the limit exists. The limit need not exist, or may be infinite, at a critical point of a smooth function \(u\) when \(p\le\frac2{2-s}\). For instance, at a nondegenerate maximum the integrand behaves like \(|x-y|^{2(p-1)-N-sp}\). For this reason, every pointwise evaluation used in this paper is justified explicitly in this subsection.

	\medskip

	Let \(u\) be bounded and \(\varphi\in C^\infty_c(\RN)\). Since the kernel is integrable away from the diagonal, Fubini's theorem and the symmetry of the domain \(\{|x-y|>\varepsilon\}\) give
	\begin{equation}\label{eq:trunc}
		\iint_{\{|x-y|>\varepsilon\}}\frac{\Psi_p(u(x)-u(y))(\varphi(x)-\varphi(y))}{|x-y|^{N+sp}}\,dx\,dy=\int_{\RN}(-\Delta_p)^s_\varepsilon u(x)\,\varphi(x)\,dx.
	\end{equation}
	If moreover \(u\) is locally Lipschitz, the left-hand side converges to \(\E_{p,s}(u,\varphi)\) as \(\varepsilon\to0^+\), by dominated convergence. Consequently,
	\begin{equation}\label{eq:Eident}
		\E_{p,s}(u,\varphi)=\int_{\RN}(-\Delta_p)^su(x)\,\varphi(x)\,dx
	\end{equation}
	as soon as \((-\Delta_p)^s_\varepsilon u\to(-\Delta_p)^su\) a.e.\ in \(\operatorname{supp}\varphi\) and there is \(g\in L^1_\loc(\RN)\) such that \(|(-\Delta_p)^s_\varepsilon u|\le g\) in \(\operatorname{supp}\varphi\) for every \(\varepsilon\in(0,1)\). To verify these two conditions we use the following two lemmas. They follow the arguments of \cite[Lemmas~3.6, 3.8 and~3.10]{KKL19}, and we include the proofs for completeness, since we need the explicit dependence of the bounds on \(|\nabla u(x)|\) and on the second derivatives of \(u\).

	\begin{lemma}[Local estimate]\label{lem:local}
		Let \(N\ge2\), \(s\in(0,1)\), \(p\in(1,\infty)\), \(x\in\RN\) and \(\rho>0\). Let \(u\in C^2(B_\rho(x))\) with bounded second derivatives, and set
		\[
			G\coloneqq|\nabla u(x)|,\qquad M\coloneqq\sup_{B_\rho(x)}|D^2u|.
		\]
		If \(1<p<2\), assume in addition that \(G>0\). Then there is a constant \(C=C(N,s,p)>0\) such that, for every \(0<\varepsilon<\rho\),
		\begin{equation}\label{eq:near}
			\left|2\int_{\{\varepsilon<|x-y|<\rho\}}\frac{\Psi_p(u(x)-u(y))}{|x-y|^{N+sp}}\,dy\right|\le C\,\Theta(G,M,\rho),
		\end{equation}
		where
		\[
			\Theta(G,M,\rho)\coloneqq
			\begin{cases}
				(G+M\rho)^{p-2}M\rho^{p-sp},&\text{if }p\ge2,\\[2pt]
				M\,G^{p-2}\rho^{p-sp}+M^p\,G^{-1}\rho^{2p-1-sp},&\text{if }1<p<2.
			\end{cases}
		\]
		Moreover, the integral on the left-hand side of \eqref{eq:near} has a finite limit as \(\varepsilon\to0^+\).
	\end{lemma}

	\begin{proof}
		Changing variables \(y=x+h\) and symmetrizing in \(h\mapsto-h\), the integral in \eqref{eq:near} equals
		\[
			\int_{\{\varepsilon<|h|<\rho\}}\frac{\Psi_p(u(x)-u(x+h))+\Psi_p(u(x)-u(x-h))}{|h|^{N+sp}}\,dh.
		\]
		By Taylor's formula, \(u(x\pm h)=u(x)\pm\nabla u(x)\cdot h+b_\pm(h)\) with \(|b_\pm(h)|\le\frac12M|h|^2\) for \(|h|<\rho\). Since \(\Psi_p\) is odd, the numerator is
		\[
			\Psi_p(\nabla u(x)\cdot h-b_-(h))-\Psi_p(\nabla u(x)\cdot h+b_+(h)).
		\]
		We estimate its absolute value, which we denote by \(I(h)\).

		\medskip

		\noindent{\bf \emph{Case \(p\ge2\)}:} By the mean value theorem, \(|\Psi_p(t)-\Psi_p(\tau)|\le(p-1)\max\{|t|,|\tau|\}^{p-2}|t-\tau|\). Since \(|\nabla u(x)\cdot h|\le G|h|\),
		\[
			I(h)\le(p-1)(G|h|+M|h|^2)^{p-2}M|h|^2\le(p-1)(G+M\rho)^{p-2}M|h|^p.
		\]
		Integrating against \(|h|^{-N-sp}\) over \(\{|h|<\rho\}\) and using \(p-sp>0\) we obtain \eqref{eq:near}.

		\medskip

		\noindent{\bf \emph{Case \(1<p<2\)}:} We split \(\{|h|<\rho\}\) into two regions. In the region where \(|\nabla u(x)\cdot h|\ge2M|h|^2\), both arguments of \(\Psi_p\) have the sign of \(\nabla u(x)\cdot h\) and absolute value at least \(\frac12|\nabla u(x)\cdot h|\), so the mean value theorem gives
		\[
			I(h)\le(p-1)(\tfrac12|\nabla u(x)\cdot h|)^{p-2}M|h|^2.
		\]
		Writing \(h=r\sigma\) with \(r=|h|\) and \(\sigma\in S^{N-1}\), and \(\nabla u(x)=G\omega\) with \(|\omega|=1\), the contribution of this region is bounded by
		\[
			C\,M\,G^{p-2}\int_{S^{N-1}}|\omega\cdot\sigma|^{p-2}\,d\sigma\int_0^\rho r^{p-1-sp}\,dr=C\,M\,G^{p-2}\rho^{p-sp},
		\]
		where the integral over the sphere is finite because \(p-2>-1\) and \(N\ge2\). In the region where \(|\nabla u(x)\cdot h|<2M|h|^2\) we use \(|\Psi_p(t)|=|t|^{p-1}\) to get \(I(h)\le2\left(\frac52M|h|^2\right)^{p-1}\). For fixed \(r\), this region corresponds to the directions \(\sigma\) with \(|\omega\cdot\sigma|<2Mr/G\), a subset of \(S^{N-1}\) of measure at most \(CMr/G\). Its contribution is therefore bounded by
		\[
			C\,M^{p-1}\frac MG\int_0^\rho r^{2(p-1)+1-1-sp}\,dr=C\,M^p\,G^{-1}\rho^{2p-1-sp},
		\]
		since \(2p-1-sp>0\). This proves \eqref{eq:near}.

		In both cases we have in fact shown that \(I(h)|h|^{-N-sp}\) is integrable on \(\{|h|<\rho\}\), so the limit as \(\varepsilon\to0^+\) exists by dominated convergence.
	\end{proof}

	Our first application covers the truncated powers used in Section~\ref{sec:liouville}, which are smooth with non-vanishing gradient in the annuli where they are tested.

	\begin{lemma}[Smooth functions with non-vanishing gradient]\label{lem:smooth}
		Let \(N\ge2\), \(p\in(1,\infty)\), \(s\in(0,1)\), let \(U\subset\RN\) be open, and let \(v:\RN\to\R\) be bounded and Lipschitz, with \(v\in C^2(U)\) and, if \(1<p<2\), \(\nabla v\ne0\) in \(U\). Then \((-\Delta_p)^sv(x)\) exists for every \(x\in U\), it is bounded on every compact subset of \(U\), and
		\[
			\E_{p,s}(v,\varphi)=\int_{\RN}(-\Delta_p)^sv(x)\,\varphi(x)\,dx\qquad\text{for every }\varphi\in C^\infty_c(U).
		\]
		In particular, if \((-\Delta_p)^sv\le0\) pointwise in \(U\), then \(\E_{p,s}(v,\varphi)\le0\) for every non-negative \(\varphi\in C^\infty_c(U)\).
	\end{lemma}

	\begin{proof}
		Let \(K\subset U\) be compact, let \(\rho\coloneqq\frac12\min\{1,\dist(K,\partial U)\}\), and let \(K_\rho\coloneqq\{y:\dist(y,K)\le\rho\}\), a compact subset of \(U\). Set \(G_K\coloneqq\min_K|\nabla v|\), which is positive if \(1<p<2\), and \(M_K\coloneqq\max_{K_\rho}|D^2v|<\infty\). For \(x\in K\) and \(\varepsilon\in(0,1)\) we split \((-\Delta_p)^s_\varepsilon v(x)\) into the contributions of \(\{|x-y|\ge\rho\}\) and of \(\{\varepsilon<|x-y|<\rho\}\). Let \(\Lambda\) be the Lipschitz constant of \(v\). Since \(|v(x)-v(y)|\le\min\{\Lambda|x-y|,2\|v\|_\infty\}\),
		\begin{equation}\label{eq:farLip}
			2\int_{\{|x-y|\ge\rho\}}\frac{|v(x)-v(y)|^{p-1}}{|x-y|^{N+sp}}\,dy\le C\left(1+\int_\rho^1r^{p-2-sp}\,dr\right)\le C\left(1+\rho^{p-1-sp}+|\log\rho|\right),
		\end{equation}
		with \(C\) depending only on \(N\), \(s\), \(p\), \(\Lambda\) and \(\|v\|_\infty\). The same bound holds for the whole of \((-\Delta_p)^s_\varepsilon v(x)\) when \(\varepsilon\ge\rho\). The contribution of \(\{\varepsilon<|x-y|<\rho\}\) is bounded by \(C\,\Theta(G,M_K,\rho)\) with \(G=|\nabla v(x)|\in[G_K,\|\nabla v\|_\infty]\), by Lemma~\ref{lem:local}. Hence the truncations are bounded on \(K\) uniformly in \(\varepsilon\), and the limit exists at every point of \(K\). Taking \(K=\operatorname{supp}\varphi\), the identity follows from \eqref{eq:trunc} and dominated convergence.
	\end{proof}

	Our second application concerns the barrier used in the existence theorem. It is smooth with non-vanishing gradient away from the origin, so the new point is its behavior near \(x=0\).

	\begin{lemma}[The barrier]\label{lem:barrier}
		Let \(N\ge2\), \(p\in(1,\infty)\), \(\kappa>0\), \(A>0\), \(\lambda\ge1\) and
		\[
			W(x)\coloneqq A(\lambda+|x|)^{-\kappa},\qquad x\in\RN.
		\]
		Then the following hold.
		\begin{enumerate}[label=\textup{(\alph*)}]
			\item The function \(W\) is bounded and Lipschitz in \(\RN\). In particular \(W\in W^{1,p}_\loc(\RN)\cap L^{p-1}_{sp}(\RN)\) for every \(s\in(0,1)\).
			\item For every \(s\in(0,1]\), the value \((-\Delta_p)^sW(x)\) exists for every \(x\ne0\), and \((-\Delta_p)^sW\in L^1_\loc(\RN)\). For \(s=1\),
			\begin{equation}\label{eq:lapW}
				-\Delta_pW(x)=(A\kappa)^{p-1}(\lambda+|x|)^{-\kappa(p-1)-p}\left[N-p-(p-1)\kappa+\frac{(N-1)\lambda}{|x|}\right],\qquad x\ne0.
			\end{equation}
			\item For every \(s\in(0,1]\) and every \(\varphi\in C^\infty_c(\RN)\),
			\[
				\E_{p,s}(W,\varphi)=\int_{\RN}(-\Delta_p)^sW(x)\,\varphi(x)\,dx.
			\]
		\end{enumerate}
		Consequently, if \(W\) satisfies \(\Lc W\ge f\) pointwise in \(\RN\setminus\{0\}\) for some \(f\in L^1_\loc(\RN)\), then \(\Lc W\ge f\) weakly in \(\RN\).
	\end{lemma}

	\begin{proof}
		Part (a) is clear, since \(|\nabla W(x)|=A\kappa(\lambda+|x|)^{-\kappa-1}\le A\kappa\) for \(x\ne0\). Note also that \(W\) is smooth in \(\RN\setminus\{0\}\) and that
		\begin{equation}\label{eq:D2W}
			|D^2W(x)|\le C_0\left(1+|x|^{-1}\right)\qquad\text{for all }x\ne0,
		\end{equation}
		with \(C_0\) depending only on \(N\), \(A\) and \(\kappa\).

		We prove (b) and (c) separately in the local and in the nonlocal case.

		\medskip

		\noindent{\bf \emph{Case 1}:} The local case \(s=1\). Write \(W=W(r)\) with \(r=|x|>0\). Then \(W'(r)=-A\kappa(\lambda+r)^{-\kappa-1}\) and \(W''(r)=A\kappa(\kappa+1)(\lambda+r)^{-\kappa-2}\), so
		\begin{align*}
			-\Delta_pW&=-|W'|^{p-2}\left((p-1)W''+\frac{N-1}rW'\right)\\
			&=(A\kappa)^{p-1}(\lambda+r)^{-(\kappa+1)(p-1)-1}\left[\frac{(N-1)(\lambda+r)}r-(p-1)(\kappa+1)\right].
		\end{align*}
		Since \((\kappa+1)(p-1)+1=\kappa(p-1)+p\) and
		\[
			\frac{(N-1)(\lambda+r)}r-(p-1)(\kappa+1)=N-p-(p-1)\kappa+\frac{(N-1)\lambda}r,
		\]
		we get \eqref{eq:lapW}. In particular \(|\Delta_pW(x)|\le C(1+|x|^{-1})\) on bounded sets, which is locally integrable because \(N\ge2\). Fix \(\varphi\in C^\infty_c(\RN)\) and \(\delta>0\). Integrating by parts in \(\RN\setminus B_\delta(0)\), where \(W\) is smooth and \(\nabla W\ne0\),
		\[
			\int_{\{|x|>\delta\}}|\nabla W|^{p-2}\nabla W\cdot\nabla\varphi\,dx=\int_{\{|x|>\delta\}}(-\Delta_pW)\,\varphi\,dx-\int_{\{|x|=\delta\}}\varphi\,|\nabla W|^{p-2}\nabla W\cdot\frac x{|x|}\,d\sigma.
		\]
		The boundary term is bounded by \(\|\varphi\|_{L^\infty}(A\kappa)^{p-1}|S^{N-1}|\delta^{N-1}\), which tends to \(0\) as \(\delta\to0^+\). Since \(|\nabla W|^{p-1}|\nabla\varphi|\) and \(|\Delta_pW||\varphi|\) are integrable, letting \(\delta\to0^+\) proves (c).

		\medskip

		\noindent{\bf \emph{Case 2}:} The nonlocal case \(s\in(0,1)\). Fix \(x\ne0\), let \(\rho\coloneqq\frac12\min\{|x|,1\}\) and \(\varepsilon\in(0,1)\). We split \((-\Delta_p)^s_\varepsilon W(x)\) into the contributions of \(\{|x-y|\ge\rho\}\) and of \(\{\varepsilon<|x-y|<\rho\}\).

		Since \(W\) is bounded and Lipschitz, the estimate \eqref{eq:farLip} in the proof of Lemma~\ref{lem:smooth} bounds the contribution of \(\{|x-y|\ge\rho\}\) by \(C(1+\rho^{p-1-sp}+|\log\rho|)\), and the same bound holds for the whole of \((-\Delta_p)^s_\varepsilon W(x)\) when \(\varepsilon\ge\rho\).

		On \(\{\varepsilon<|x-y|<\rho\}\) we apply Lemma~\ref{lem:local} in the ball \(B_\rho(x)\), which does not contain the origin. Here \(G=A\kappa(\lambda+|x|)^{-\kappa-1}\) and, by \eqref{eq:D2W}, \(M\le C_0(1+2|x|^{-1})\). If \(|x|\le1\), then \(\rho=|x|/2\), the number \(G\) lies between two positive constants, and \(M\le C|x|^{-1}\). Inserting these bounds in the definition of \(\Theta\) we find, in both cases \(p\ge2\) and \(1<p<2\),
		\[
			\Theta(G,M,\rho)\le C|x|^{p-1-sp}.
		\]
		If \(|x|\ge1\), then \(\rho=\frac12\), \(M\le3C_0\), and \(G\) is bounded from below by a positive constant on every bounded set, so \(\Theta(G,M,\rho)\) is bounded on bounded subsets of \(\{|x|\ge1\}\).

		By Lemma~\ref{lem:local}, \((-\Delta_p)^sW(x)\) exists for every \(x\ne0\), and for every bounded set \(B\subset\RN\) there is a constant \(C_B\) such that
		\[
			\sup_{0<\varepsilon<1}|(-\Delta_p)^s_\varepsilon W(x)|\le g(x)\coloneqq C_B\left(1+|x|^{p-1-sp}+|\log|x||\right)\qquad\text{for }x\in B\setminus\{0\}.
		\]
		The function \(g\) is integrable on \(B\), because \(p-1-sp>-N\). This proves (b), and (c) follows from \eqref{eq:trunc} and dominated convergence, as explained before Lemma~\ref{lem:local}.

		\medskip

		For the last assertion, we apply (c) to each operator and add. For every non-negative \(\varphi\in C^\infty_c(\RN)\) we obtain
		\[
			\E_\Lc(W,\varphi)=\int_{\RN}\Lc W\,\varphi\,dx\ge\int_{\RN}f\,\varphi\,dx.\qedhere
		\]
	\end{proof}

	Our last application concerns a cut-off function which is flat to infinite order at the points where its gradient vanishes. Let \(\Phi(t)\coloneqq e^{-1/t}\) for \(t>0\) and \(\Phi(t)\coloneqq0\) for \(t\le0\), and set
	\begin{equation}\label{eq:mu}
		\mu(t)\coloneqq\frac{\Phi(1-t)}{\Phi(1-t)+\Phi(t-\frac12)},\qquad t\in\R.
	\end{equation}
	The denominator is positive for every \(t\), so \(\mu\in C^\infty(\R)\). Moreover \(0\le\mu\le1\), \(\mu=1\) on \((-\infty,\frac12]\), \(\mu=0\) on \([1,\infty)\), and
	\[
		\mu'(t)=-\frac{\Phi'(1-t)\,\Phi\left(t-\frac12\right)+\Phi(1-t)\,\Phi'\left(t-\frac12\right)}{\left(\Phi(1-t)+\Phi(t-\frac12)\right)^2}<0\qquad\text{for }t\in(\tfrac12,1).
	\]
	All the derivatives of \(\mu\) vanish at \(t=\frac12\) and at \(t=1\). We shall use two quantitative forms of this flatness. First, for every \(\ell\ge1\) there is \(C_\ell>0\) such that
	\begin{equation}\label{eq:flat}
		|\mu(t)-\mu(t_0)|\le C_\ell|t-t_0|^\ell\qquad\text{for }t_0\in\{\tfrac12,1\}\text{ and all }t\in\R,
	\end{equation}
	by Taylor's formula for \(|t-t_0|\le1\) and because \(\mu\) is bounded. Second, writing
	\[
		\delta(t)\coloneqq\min\{t-\tfrac12,\,1-t\}\qquad\text{for }t\in(\tfrac12,1),
	\]
	there is a constant \(C_\mu\ge1\) such that, for every \(t\in(\frac12,1)\),
	\begin{equation}\label{eq:muP}
		C_\mu^{-1}\,\delta(t)^{-2}e^{-1/\delta(t)}\le|\mu'(t)|\le C_\mu\,\delta(t)^{-2}e^{-1/\delta(t)},\qquad|\mu''(t)|\le C_\mu\,\delta(t)^{-4}e^{-1/\delta(t)}.
	\end{equation}
	Indeed, for \(t\in[\frac34,1)\) we have \(\delta(t)=1-t\), \(\Phi'(1-t)=\delta(t)^{-2}e^{-1/\delta(t)}\) and \(|\Phi''(1-t)|\le\delta(t)^{-4}e^{-1/\delta(t)}\), while \(\Phi(t-\frac12)\) lies between \(e^{-4}\) and \(e^{-2}\) and has bounded derivatives. The two terms in the numerator of \(\mu'\) are positive, the first one gives the lower bound, and the upper bounds follow from the quotient rule. The interval \((\frac12,\frac34]\) is reduced to the previous one by the symmetry \(\mu(\frac32-t)=1-\mu(t)\).

	For \(p\ge2\), and for \(1<p<2\) with \(s<\frac{2(p-1)}p\), the boundedness of \((-\Delta_p)^s\mu(|\cdot|)\) follows from \cite[Proposition~2.12]{IMS16}. The flatness of \(\mu\) is what allows us to cover every \(p\in(1,\infty)\).

	\begin{lemma}[Flat profiles]\label{lem:flat}
		Let \(N\ge2\) and \(\nu\in\mathbb N\), and let \(c_0\in\R\), \(c_j>0\), \(\rho_j>0\) and \(z_j\in\RN\) for \(j\in\{1,\dots,\nu\}\), with the closed balls \(\overline{B_{\rho_j}(z_j)}\) pairwise disjoint. Let
		\[
			v(x)\coloneqq c_0+\sum_{j=1}^{\nu}c_j\,\mu\!\left(\frac{|x-z_j|}{\rho_j}\right),\qquad x\in\RN,
		\]
		where \(\mu\) is the cut-off function defined in \eqref{eq:mu}.
		Let \(p\in(1,\infty)\) and \(s\in(0,1]\).
		\begin{enumerate}[label=\textup{(\alph*)}]
			\item If \(s=1\), the vector field \(|\nabla v|^{p-2}\nabla v\) is of class \(C^1\) in \(\RN\). Hence \(-\Delta_pv\) is continuous and bounded in \(\RN\).
			\item If \(s\in(0,1)\), then \((-\Delta_p)^sv(x)\) exists for every \(x\in\RN\), and
			\[
				\sup_{x\in\RN,\ \varepsilon>0}|(-\Delta_p)^s_\varepsilon v(x)|<\infty.
			\]
		\end{enumerate}
		In both cases \((-\Delta_p)^sv\in L^\infty(\RN)\), and
		\begin{equation}\label{eq:Eflat}
			\E_{p,s}(v,\varphi)=\int_{\RN}(-\Delta_p)^sv(x)\,\varphi(x)\,dx\qquad\text{for every }\varphi\in C^\infty_c(\RN).
		\end{equation}
	\end{lemma}

	\begin{proof}
		We first reduce the proof to one profile. In case (a) the fields \(|\nabla v_j|^{p-2}\nabla v_j\) of the functions \(v_j\coloneqq\mu(|\cdot-z_j|/\rho_j)\) have disjoint supports, so it is enough to consider one of them. In case (b), let \(2d_0>0\) be the smallest distance between two of the balls (\(d_0\coloneqq1\) if \(\nu=1\)). For every \(x\in\RN\) the ball \(B_{d_0}(x)\) meets at most one of them, say the \(j\)-th one, so \(v=c_0+c_jv_j\) in \(B_{d_0}(x)\). Since \(v\) and \(v_j\) are bounded, the contributions of \(\{|x-y|\ge d_0\}\) to \((-\Delta_p)^s_\varepsilon v(x)\) and to \((-\Delta_p)^s_\varepsilon(c_jv_j)(x)\) are bounded by a constant independent of \(x\) and \(\varepsilon\), and the contributions of \(\{|x-y|<d_0\}\) coincide. Hence it suffices to prove (b) for each \(c_jv_j\), and by \eqref{eq:scaling} we may assume in both cases that
		\[
			v(x)=\mu(|x|).
		\]

		\medskip

		\noindent{\bf \emph{Case \(s=1\)}:} In the annulus \(\{\frac12<|x|<1\}\) we have \(|\nabla v|^{p-2}\nabla v=-|\mu'(|x|)|^{p-1}\frac x{|x|}\), and \(\nabla v=0\) elsewhere. The function \(t\mapsto|\mu'(t)|^{p-1}\) is smooth in \((\frac12,1)\), since \(\mu'<0\) there, and by \eqref{eq:muP}
		\begin{align*}
			|\mu'(t)|^{p-1}&\le C\delta(t)^{-2(p-1)}e^{-(p-1)/\delta(t)},\\
			\left|\frac d{dt}|\mu'(t)|^{p-1}\right|&=(p-1)|\mu'(t)|^{p-2}|\mu''(t)|\le C\delta(t)^{-2p}e^{-(p-1)/\delta(t)}.
		\end{align*}
		For \(1<p<2\) the lower bound in \eqref{eq:muP} has been used. Both quantities tend to \(0\) as \(\delta(t)\to0^+\), so the function equal to \(|\mu'|^{p-1}\) in \((\frac12,1)\) and to \(0\) elsewhere is of class \(C^1\) in \(\R\). Hence the field is of class \(C^1\) in \(\RN\), and \eqref{eq:Eflat} follows from the divergence theorem.

		\medskip

		\noindent{\bf \emph{Case \(s\in(0,1)\)}:} Since \(0\le v\le1\), for every \(x\in\RN\), \(\varepsilon>0\) and \(\tau>0\),
		\begin{equation}\label{eq:far}
			2\int_{\{|x-y|\ge\tau\}}\frac{|v(x)-v(y)|^{p-1}}{|x-y|^{N+sp}}\,dy\le C\tau^{-sp}.
		\end{equation}
		Fix \(\ell\ge1\) with \(\ell(p-1)>sp\). We distinguish three cases according to the position of \(x\).

		\medskip

		\noindent{\bf \emph{Case 1}:} Points with \(|x|\le\frac12\) or \(|x|\ge1\). Suppose \(|x|\le\frac12\). If \(|y|\le\frac12\) then \(v(x)-v(y)=0\), and if \(|y|>\frac12\) then, by \eqref{eq:flat},
		\[
			|v(x)-v(y)|=|\mu(\tfrac12)-\mu(|y|)|\le C_\ell(|y|-\tfrac12)^\ell\le C_\ell|x-y|^\ell.
		\]
		The case \(|x|\ge1\) is analogous. Thus \(|v(x)-v(y)|\le\min\{1,C_\ell|x-y|^\ell\}\) for every \(y\), and the integrand of \((-\Delta_p)^s_\varepsilon v(x)\) is dominated by \(\min\{1,C_\ell|x-y|^\ell\}^{p-1}|x-y|^{-N-sp}\), whose integral over \(\RN\) is finite and independent of \(x\).

		\medskip

		\noindent{\bf \emph{Case 2}:} Points with \(\frac12<|x|<1\) and \(\delta(|x|)\ge d_1\), where \(d_1\in(0,\frac18)\) is a fixed number. On the compact set of such points, \(G=|\mu'(|x|)|\) is bounded from below by a positive constant and the second derivatives of \(v\) are bounded in \(B_{d_1/2}(x)\). We apply Lemma~\ref{lem:local} with \(\rho=d_1/2\) and \eqref{eq:far} with \(\tau=d_1/2\).

		\medskip

		\noindent{\bf \emph{Case 3}:} Points with \(\frac12<|x|<1\) and \(\delta\coloneqq\delta(|x|)<d_1\). We split \(\RN\) into the regions \(\{|x-y|<\delta^2\}\), \(\{\delta^2\le|x-y|<\frac\delta2\}\) and \(\{|x-y|\ge\frac\delta2\}\).

		In the first region we apply Lemma~\ref{lem:local} with \(\rho=\delta^2\). For \(y\in B_{\delta^2}(x)\) we have \(\delta(|y|)\in[\delta-\delta^2,\delta+\delta^2]\) and
		\[
			\frac1\delta-\frac1{\delta+\delta^2}=\frac1{1+\delta}\le1,
		\]
		so \eqref{eq:muP} gives \(G\ge C^{-1}\delta^{-2}e^{-1/\delta}\), \(G\le C\delta^{-2}e^{-1/\delta}\) and \(M\le C\delta^{-4}e^{-1/\delta}\). Inserting these bounds in the definition of \(\Theta\) we get
		\[
			\Theta(G,M,\delta^2)\le C\,\delta^{-4p}e^{-(p-1)/\delta},
		\]
		which is bounded for \(\delta\in(0,d_1)\).

		In the second region, every point \(y\) satisfies \(\delta(|y|)\in[\frac\delta2,\frac{3\delta}2]\), so \(|\nabla v|\le C\delta^{-2}e^{-2/(3\delta)}\) along the segment from \(x\) to \(y\), and \(|v(x)-v(y)|\le C\delta^{-2}e^{-2/(3\delta)}|x-y|\). The contribution of this region is therefore bounded by
		\[
			C\left(\delta^{-2}e^{-2/(3\delta)}\right)^{p-1}\int_{\delta^2}^1r^{p-2-sp}\,dr\le C\,\delta^{-2(p-1)-2sp}e^{-2(p-1)/(3\delta)},
		\]
		which is again bounded for \(\delta\in(0,d_1)\).

		In the third region, let \(x_*\) be the point of the nearer of the two spheres \(\{|y|=\frac12\}\), \(\{|y|=1\}\) closest to \(x\), so that \(|x-x_*|=\delta\le2|x-y|\). By \eqref{eq:flat},
		\[
			|v(x)-v(y)|\le|v(x)-v(x_*)|+|v(x_*)-v(y)|\le C_\ell\delta^\ell+C_\ell|x_*-y|^\ell\le C|x-y|^\ell,
		\]
		and we conclude as in Case 1.

		In the three cases the truncations \((-\Delta_p)^s_\varepsilon v(x)\) are bounded uniformly in \(x\) and \(\varepsilon\), and the integrand is absolutely integrable near \(y=x\) after symmetrization, so the limit as \(\varepsilon\to0^+\) exists. Finally, \eqref{eq:Eflat} follows from \eqref{eq:trunc} and dominated convergence.
	\end{proof}

	By Lemma~\ref{lem:flat}, the constant
	\begin{equation}\label{eq:Cstar}
		C^*\coloneqq\max_{1\le i\le n}\sup_{\RN}\left|(-\Delta_{p_i})^{s_i}\mu(|\cdot|)\right|
	\end{equation}
	is finite and depends only on \(N\) and \(\{(p_i,s_i)\}_{i=1}^n\).

\subsection{Comparison principle and strong minimum principle}\label{subsec:comparison}
	Finally, we establish a weak comparison principle and a strong minimum principle for \(\Lc\). Both are applied in the Liouville argument of Section~\ref{sec:liouville}. Recall the splitting \eqref{eq:FL} of the indices into fractional and local ones.

	\begin{proposition}[Comparison principle]\label{prop:comparison}
		Let \(\Omega\subset\RN\) be a bounded open set, and let \(u,v\in\DL(\Omega)\) be continuous in \(\overline\Omega\). Assume that \(\Lc u\ge\Lc v\) weakly in \(\Omega\) and that \(u\ge v\) in \(\RN\setminus\Omega\). Then \(u\ge v\) in \(\RN\).
	\end{proposition}

	\begin{proof}
		The proof is standard, see for instance \cite{DPQ17} for the fractional \(p\)-Laplacian and \cite{LG26} for the mixed operator \(-\Delta_p+(-\Delta_p)^s\). The argument is the same for every summand of \(\Lc\).
	\end{proof}

	\begin{proposition}[Strong minimum principle]\label{prop:smp}
		Let \(u\in\DL(\RN)\cap C(\RN)\) be a non-negative weak super-solution of \(\Lc u=0\) in \(\RN\). If \(u(x_0)=0\) for some \(x_0\in\RN\), then \(u\equiv0\) in \(\RN\).
	\end{proposition}

	\begin{proof}
		Suppose that \(u(x_0)=0\) and \(u\not\equiv0\). By continuity there are \(z\in\RN\), \(\rho>0\) and \(a>0\) such that \(u\ge a\) in \(\overline{B_\rho(z)}\) and \(|z-x_0|>2\rho\). In both cases below we build \(v\) with \(\Lc v\le0\) weakly in a bounded open set \(\Omega\ni x_0\), \(v\le u\) in \(\RN\setminus\Omega\) and \(v(x_0)>0\). Proposition~\ref{prop:comparison} then gives \(u(x_0)\ge v(x_0)>0\), a contradiction.

		\medskip

		\noindent{\bf \emph{Case \(F\ne\emptyset\)}:} Fix \(r\in(0,\rho)\), let \(d\) and \(D\) be the smallest and the largest distance between a point of \(B_r(x_0)\) and a point of \(B_\rho(z)\), and, for \(\delta\in(0,a/2]\), let
		\[
			v(x)\coloneqq\delta\,\mu\!\left(\frac{|x-x_0|}r\right)+a\,\mu\!\left(\frac{|x-z|}\rho\right),\qquad\Omega\coloneqq B_r(x_0).
		\]
		By Lemma~\ref{lem:flat}, \(\Lc v\) exists pointwise and \eqref{eq:Eflat} holds. Let \(x\in B_r(x_0)\). For \(i\in F\), the first profile contributes at most \(C^*\delta^{p_i-1}r^{-s_ip_i}\), by \eqref{eq:scaling}. Since \(v(x)\le\delta\), the second profile contributes at most \(2\delta^{p_i-1}|B_\rho|d^{-N-s_ip_i}\) from \(B_\rho(z)\setminus B_{\rho/2}(z)\) and at most \(-2(a/2)^{p_i-1}|B_{\rho/2}|D^{-N-s_ip_i}\) from \(B_{\rho/2}(z)\), where \(v=a\ge v(x)+a/2\). For \(i\in L\), \(v\) coincides with the first profile near \(x\), so \(-\Delta_{p_i}v(x)\le C^*\delta^{p_i-1}r^{-p_i}\). The negative terms do not depend on \(\delta\), so \(\Lc v\le0\) in \(B_r(x_0)\) for \(\delta\) small. Moreover \(v\le a\le u\) in \(B_\rho(z)\) and \(v=0\) elsewhere outside \(B_r(x_0)\), and \(v(x_0)=\delta\).

		\medskip

				\noindent{\bf \emph{Case \(F=\emptyset\)}:} Fix \(R>|x_0-z|\) and
		\[
			\beta<\min\left\{0,\ \min_{1\le i\le n}\frac{p_i-N}{p_i-1}\right\}.
		\]
		For every \(i\) we have \(\beta<0\) and \((p_i-1)\beta+N-p_i<0\), so \eqref{eq:c1p} gives 
        \[
        \cc_{1,p_i}(\beta)=-\Psi_{p_i}(\beta)\left((p_i-1)\beta+N-p_i\right)<0.
        \]
       Note that \eqref{eq:c1p} is a direct computation for radial functions, and it does not use \(p_i<N\). Let \(\Omega\coloneqq\{\rho<|x-z|<R\}\), \(\alpha\coloneqq a/(\rho^\beta-R^\beta)\), which is positive because \(\beta<0\), and let \(v\) be equal to \(a\) in \(\overline{B_\rho(z)}\), to \(\alpha(|x-z|^\beta-R^\beta)\) in \(\Omega\), and to \(0\) outside \(B_R(z)\). By \eqref{eq:c1p} and an integration by parts, \(\E_\Lc(v,\varphi)=\sum_i\alpha^{p_i-1}\cc_{1,p_i}(\beta)\int_\Omega|x-z|^{\beta(p_i-1)-p_i}\varphi\,dx\le0\) for every non-negative \(\varphi\in C^\infty_c(\Omega)\). Clearly \(v\le u\) outside \(\Omega\), and \(v(x_0)>0\) because \(\rho<|x_0-z|<R\) and \(\beta<0\).
	\end{proof}

	\begin{remark}
		For a single fractional \(p\)-Laplacian, Proposition~\ref{prop:smp} is contained in \cite[Theorem~1.2]{DPQ17}. When \(F=\emptyset\) and \(u\) is of class \(C^1\), it is contained in the classical strong maximum principle for quasilinear operators in divergence form, see \cite{PS07,Vaz84}. The argument above only requires continuity.
	\end{remark}
%=====================================================================
\section{Fundamental solutions and barrier estimates}\label{sec:barrier}

	We begin with three technical results that will be used throughout the paper. We first recall the formula satisfied by the powers \(|x|^\beta\) for the local and nonlocal \(p\)-Laplacian. The second result establishes that both the local and nonlocal \(p\)-Laplacians of the barrier \(w_\kappa(x)=(1+|x|)^{-\kappa}\), with \(\kappa>0\), admit an explicit signed lower bound in terms of a power of \(1+|x|\), whose constant may have either sign. The third one gives a cruder lower bound in the range \(\kappa\ge N/(p-1)\), which is not covered by the second one. These estimates are the key ingredients of our existence result, and Lemma~\ref{lem:fundamental} is also used for the barriers of Section~\ref{sec:liouville}.

\subsection{Fundamental constant}
	For a single operator of order \(s\in(0,1]\) and \(\beta=\frac{sp-N}{p-1}\), the power \(v_\beta(x)=|x|^\beta\) is, up to a multiplicative constant, the fundamental solution, and the sign of the constant produced by a general power selects the admissible barriers. We collect the local and nonlocal cases in one statement, and refer to \cite{KV86} for the classical local computation and to \cite[Theorem~1.1 and Remark~1.1]{DPQ26} for the nonlocal one, where the constant is computed for the same operator as ours, without normalization constant.

	\begin{lemma}[Fundamental constant]\label{lem:fundamental}
		Let \(N\ge2\), \(p\in(1,\infty)\) and \(s\in(0,1]\) with \(sp<N\). For every \(\beta\in\left(-\frac N{p-1},\frac{sp}{p-1}\right)\) the function \(v_\beta(x)=|x|^\beta\) satisfies
		\begin{equation}\label{eq:fund}
			(-\Delta_p)^sv_\beta(x)=\cc_{s,p}(\beta)\,|x|^{\beta(p-1)-sp}\qquad\text{for every }x\in\RN\setminus\{0\},
		\end{equation}
		where \(\cc_{s,p}(\beta)\coloneqq(-\Delta_p)^sv_\beta(e_1)\in\R\) and \(e_1=(1,0,\dots,0)\). For \(s=1\),
		\begin{equation}\label{eq:c1p}
			\cc_{1,p}(\beta)=-\Psi_p(\beta)((p-1)\beta+N-p),
		\end{equation}
		and in this case \eqref{eq:fund} holds for every \(\beta\in\R\). In both cases the sign of the constant is
		\begin{equation}\label{eq:csign}
			\cc_{s,p}(\beta)
			\begin{cases}
				=0,&\text{if }\beta=0\text{ or }\beta=\frac{sp-N}{p-1},\\[2pt]
				>0,&\text{if }\frac{sp-N}{p-1}<\beta<0,\\[2pt]
				<0,&\text{otherwise.}
			\end{cases}
		\end{equation}
		In particular, for \(\beta=-\kappa\) with \(0<\kappa<\frac N{p-1}\), the constant \(\cc_{s,p}(-\kappa)\) is positive if \(0<\kappa<\frac{N-sp}{p-1}\), and \(\cc_{1,p}(-\kappa)=\kappa^{p-1}(N-p-(p-1)\kappa)\) for every \(\kappa>0\).
	\end{lemma}

\subsection{The pointwise bound}
	Throughout we write
	\[
		w_\kappa(x)\coloneqq(1+|x|)^{-\kappa},\qquad x\in\RN,\ \kappa>0.
	\]
	The local and the nonlocal orders yield the \emph{same} pointwise lower bound, so we state them together and split the proof into the two cases.

	\begin{proposition}[Pointwise bound]\label{prop:pointwise}
		Let \(N\ge2\), \(p\in(1,\infty)\), \(s\in(0,1]\) with \(sp<N\), and \(\kappa>0\). If \(s\in(0,1)\), assume in addition that \(\kappa<\frac N{p-1}\). Then the function \(w_\kappa\) satisfies
		\begin{equation}\label{eq:pointwise}
			(-\Delta_p)^sw_\kappa\ge\cc_{s,p}(-\kappa)\,w_\kappa^{Q(\kappa)}\qquad\text{in }\RN\setminus\{0\},
		\end{equation}
		where the exponent is defined by
		\[
			Q(\kappa)\coloneqq p-1+\tfrac{sp}\kappa,
		\]
		and \(\cc_{s,p}(-\kappa)\) is the constant of Lemma~\ref{lem:fundamental}. The constant \(\cc_{s,p}(-\kappa)\) is positive for \(\kappa<\frac{N-sp}{p-1}\), zero for \(\kappa=\frac{N-sp}{p-1}\), and negative for \(\kappa>\frac{N-sp}{p-1}\).
	\end{proposition}

	\begin{proof}
		We prove the local and the nonlocal case separately.

		\medskip

		\noindent{\bf \emph{Case 1}:} The local case \(s=1\). By \eqref{eq:lapW} with \(A=\lambda=1\),
		\[
			-\Delta_pw_\kappa(x)=\kappa^{p-1}(1+|x|)^{-\kappa(p-1)-p}\left[N-p-(p-1)\kappa+\tfrac{N-1}{|x|}\right]\qquad\text{for }x\in\RN\setminus\{0\}.
		\]
		Since \(\frac{N-1}{|x|}>0\), we deduce \eqref{eq:pointwise} with \(sp=p\) and \(\cc_{1,p}(-\kappa)=\kappa^{p-1}(N-p-(p-1)\kappa)\), for every \(\kappa>0\) and whatever the sign of this constant.

		\medskip

		\noindent{\bf \emph{Case 2}:} The nonlocal case \(s\in(0,1)\). Fix \(x\in\RN\setminus\{0\}\). By Lemma~\ref{lem:barrier}, the value \((-\Delta_p)^sw_\kappa(x)\) exists, that is,
		\[
			(-\Delta_p)^sw_\kappa(x)=\lim_{\varepsilon\to0^+}2\int_{\RN\setminus B_\varepsilon(x)}\frac{\Psi_p(w_\kappa(x)-w_\kappa(y))}{|x-y|^{N+sp}}\,dy.
		\]
		The \((p-1)\)-homogeneity \(\Psi_p(\alpha t)=\alpha^{p-1}\Psi_p(t)\) for \(\alpha\ge0\) gives
		\[
			\Psi_p(w_\kappa(x)-w_\kappa(y))=\frac1{(1+|x|)^{\kappa(p-1)}}\,\Psi_p\left(1-\left(\frac{1+|x|}{1+|y|}\right)^\kappa\right).
		\]
		By rotational invariance we may assume \(x/|x|=e_1\). The change of variables \(z=\frac1{1+|x|}(y+e_1)\), for which \(dy=(1+|x|)^Ndz\) and \(x-y=(1+|x|)(e_1-z)\), yields
		\begin{equation}\label{eq:rescaled}
			(1+|x|)^{\kappa(p-1)+sp}\,2\int_{\RN\setminus B_\varepsilon(x)}\frac{\Psi_p(w_\kappa(x)-w_\kappa(y))}{|x-y|^{N+sp}}\,dy=2\int_{\RN\setminus B_{\varepsilon'}(e_1)}\frac{\Psi_p\left(1-\left(\frac{1+|x|}{1+|(1+|x|)z-e_1|}\right)^\kappa\right)}{|e_1-z|^{N+sp}}\,dz,
		\end{equation}
		where \(\varepsilon'\coloneqq\varepsilon/(1+|x|)\). The triangle inequality \((1+|x|)|z|\le|(1+|x|)z-e_1|+1\) gives
		\begin{equation}\label{eq:tri}
			\frac{1+|x|}{1+|(1+|x|)z-e_1|}\le\frac1{|z|}\qquad\text{for every }z\ne0,
		\end{equation}
		and since \(\kappa>0\) and \(\Psi_p\) is non-decreasing, we obtain
		\[
			\Psi_p\left(1-\left(\frac{1+|x|}{1+|(1+|x|)z-e_1|}\right)^\kappa\right)\ge\Psi_p\left(1-|z|^{-\kappa}\right)\qquad\text{for every }z\ne0.
		\]
		The right-hand side is integrable against \(|e_1-z|^{-N-sp}\) on \(\RN\setminus B_{\varepsilon'}(e_1)\). Indeed, near \(z=0\) it behaves like \(-|z|^{-\kappa(p-1)}\), which is integrable precisely because \(\kappa(p-1)<N\), and it is bounded at infinity. Therefore the right-hand side of \eqref{eq:rescaled} is bounded from below by
		\[
			2\int_{\RN\setminus B_{\varepsilon'}(e_1)}\frac{\Psi_p(1-|z|^{-\kappa})}{|e_1-z|^{N+sp}}\,dz=2\int_{\RN\setminus B_{\varepsilon'}(e_1)}\frac{\Psi_p\left(v_{-\kappa}(e_1)-v_{-\kappa}(z)\right)}{|e_1-z|^{N+sp}}\,dz.
		\]
		Both truncations are taken over the same set, so we may let \(\varepsilon\to0^+\). Since \(-\kappa\in\left(-\frac N{p-1},0\right)\), Lemma~\ref{lem:fundamental} shows that the last integral converges to \((-\Delta_p)^sv_{-\kappa}(e_1)=\cc_{s,p}(-\kappa)\), and \eqref{eq:pointwise} follows.
	\end{proof}

	\begin{remark}[Behavior at the origin]\label{rem:origin}
		The barrier \(w_\kappa\) is smooth in \(\RN\setminus\{0\}\), but it is only Lipschitz at \(x=0\). Near the origin \((-\Delta_p)^sw_\kappa\) may be unbounded. For instance, for \(s\in(0,1)\) the integral defining \((-\Delta_p)^sw_\kappa(0)\) diverges to \(+\infty\) when \(sp\ge p-1\), and for \(s=1\) the expression \eqref{eq:lapW} behaves like \(|x|^{-1}\). This is irrelevant for our purposes, since by Lemma~\ref{lem:barrier} (with \(A=\lambda=1\)) a pointwise inequality in \(\RN\setminus\{0\}\) yields the corresponding weak inequality in all of \(\RN\).
	\end{remark}

\subsection{Large barrier exponents}
	When \(s\in(0,1)\) and \(\kappa\ge\frac N{p-1}\), the power \(|z|^{-\kappa(p-1)}\) is not integrable at the origin, and the comparison in the proof of Proposition~\ref{prop:pointwise} breaks down. In this range \(\kappa>\frac{N-sp}{p-1}\), so the operator applied to \(w_\kappa\) is expected to be negative somewhere. The following bound shows that it is nevertheless small at infinity: its main contribution comes from the mass of \(w_\kappa^{p-1}\) near the origin, and it decays like \(|x|^{-N-sp}\), up to a logarithm.

	\begin{proposition}[Large exponents]\label{prop:large}
		Let \(N\ge2\), \(p\in(1,\infty)\), \(s\in(0,1)\) with \(sp<N\), and \(\kappa\ge\frac N{p-1}\). Then there is a constant \(C=C(N,s,p,\kappa)>0\) such that
		\begin{equation}\label{eq:large}
			(-\Delta_p)^sw_\kappa(x)\ge-C\,(1+|x|)^{-(N+sp)}\log(e+|x|)\qquad\text{for every }x\in\RN\setminus\{0\}.
		\end{equation}
	\end{proposition}

	\begin{proof}
		Fix \(x\ne0\), assume \(x/|x|=e_1\), and write \(R\coloneqq1+|x|>1\) and
		\[
			t(z)\coloneqq\frac R{1+|Rz-e_1|},\qquad z\in\RN.
		\]
		Identity \eqref{eq:rescaled} does not depend on the range of \(\kappa\). For \(\varepsilon'<\frac12\) the ball \(B_{\varepsilon'}(e_1)\) is contained in \(D_1\coloneqq\{|z|\ge\frac12\}\), and we split \(\RN\setminus B_{\varepsilon'}(e_1)\) into \(D_1\setminus B_{\varepsilon'}(e_1)\) and \(D_2\coloneqq\{|z|<\frac12\}\).

		On \(D_1\) we use \eqref{eq:tri} as in the proof of Proposition~\ref{prop:pointwise}: \(\Psi_p(1-t(z)^\kappa)\ge\Psi_p(g(z))\) with \(g(z)\coloneqq1-|z|^{-\kappa}\). The function \(g\) is smooth in \(\{|z|>\frac14\}\), with \(g(e_1)=0\) and \(\nabla g(e_1)=\kappa e_1\ne0\). Let \(\hat g\colon\RN\to\R\) be bounded and Lipschitz, with \(\hat g=g\) in \(\{|z|\ge\frac14\}\). Since \(g(e_1)=0\), for \(\varepsilon'<\frac12\) we have
		\[
			2\int_{D_1\setminus B_{\varepsilon'}(e_1)}\frac{\Psi_p(g(z))}{|e_1-z|^{N+sp}}\,dz=-(-\Delta_p)^s_{\varepsilon'}\hat g(e_1)+2\int_{D_2}\frac{\Psi_p(\hat g(e_1)-\hat g(z))}{|e_1-z|^{N+sp}}\,dz.
		\]
		The last integral is finite, because \(\hat g\) is bounded and \(|e_1-z|\ge\frac12\) in \(D_2\), and the first term converges as \(\varepsilon'\to0^+\) by Lemma~\ref{lem:smooth}, applied with \(U\) a small ball centered at \(e_1\). Hence the limit
		\[
			K_1\coloneqq\lim_{\varepsilon'\to0^+}2\int_{D_1\setminus B_{\varepsilon'}(e_1)}\frac{\Psi_p(g(z))}{|e_1-z|^{N+sp}}\,dz
		\]
		exists and is finite. It depends only on \(N\), \(s\), \(p\) and \(\kappa\).

		On \(D_2\) we use \(\Psi_p(1-t^\kappa)\ge-t^{\kappa(p-1)}\) and \(|e_1-z|\ge\frac12\). With the change of variables \(y=Rz-e_1\), for which \(|y|<\frac R2+1<2R\) when \(|z|<\frac12\), we get
		\begin{align*}
			2\int_{D_2}\frac{\Psi_p(1-t(z)^\kappa)}{|e_1-z|^{N+sp}}\,dz&\ge-2^{N+sp+1}R^{\kappa(p-1)-N}\int_{\{|y|<2R\}}(1+|y|)^{-\kappa(p-1)}\,dy\\
			&\ge-C\,R^{\kappa(p-1)-N}\log(e+R),
		\end{align*}
		where we used \((1+|y|)^{-\kappa(p-1)}\le(1+|y|)^{-N}\), since \(\kappa(p-1)\ge N\).

		The integral over \(D_2\) does not depend on \(\varepsilon'\). Letting \(\varepsilon\to0^+\) in \eqref{eq:rescaled}, which is possible by Lemma~\ref{lem:barrier}, we obtain
		\[
			R^{\kappa(p-1)+sp}(-\Delta_p)^sw_\kappa(x)\ge K_1-C\,R^{\kappa(p-1)-N}\log(e+R).
		\]
		Dividing by \(R^{\kappa(p-1)+sp}\) and using \(R^{-\kappa(p-1)-sp}\le R^{-N-sp}\), again because \(\kappa(p-1)\ge N\), we get
		\[
			(-\Delta_p)^sw_\kappa(x)\ge-(|K_1|+C)R^{-N-sp}\log(e+R),
		\]
		which gives \eqref{eq:large}, since \(\log(e+R)\le2\log(e+|x|)\).
	\end{proof}

	\begin{remark}
		If \(\kappa(p-1)>N\), the integral over \(\{|y|<2R\}\) is bounded and the logarithm in \eqref{eq:large} can be dropped. In both cases the decay rate \(N+sp\) is larger than \(N\), and this is all we need: by \eqref{eq:key}, the term that we keep in the existence argument decays like \(|x|^{-a}\) with \(a<N\).
	\end{remark}

%=====================================================================
\section{Existence of super-solutions}\label{sec:existence}

	In this section we prove part (ii) of Theorem~\ref{thm:main}, in the more precise form of Theorem~\ref{thm:existence} below. Recall the decay exponents \(\kappa_i\) from \eqref{eq:kappai}, which are positive since \(s_ip_i<N\), the functions \(Q_i\) and the number \(\kb\) from \eqref{eq:Qi}, and \(\qL\) from \eqref{eq:qL}. We also set
	\[
		\ku\coloneqq\min_{1\le i\le n}\kappa_i,\qquad\qz\coloneqq\min_{1\le i\le n}Q_i(\ku),
	\]
	for the threshold reached by the one-parameter barriers \(A(1+|x|)^{-\kappa}\) with \(\kappa\le\ku\) (Remark~\ref{rem:oneparam}).

\subsection{The thresholds}

	\begin{lemma}[Thresholds]\label{lem:thresholds}
		For every \(i\in\{1,\dots,n\}\),
		\[
			\kb\,Q_i(\kb)-N=(p_i-1)(\kb-\kappa_i)\ \ge\ 0,
		\]
		with equality if and only if \(\kappa_i=\kb\). Hence \(\min_iQ_i(\kb)=N/\kb=\qL\), and the minimum is attained exactly by the operators with the largest decay exponent.
	\end{lemma}

	\begin{proof}
		Since \(N-s_ip_i=(p_i-1)\kappa_i\), we have \(\kb\,Q_i(\kb)-N=\kb(p_i-1)+s_ip_i-N=(p_i-1)(\kb-\kappa_i)\).
	\end{proof}

\subsection{Proof of the existence theorem}
	The idea was explained in the introduction. We insert the rescaled barrier \(W(x)=A(\lambda+|x|)^{-\kappa}\) in \(\Lc\) and keep one summand with a positive constant that decays slowly enough. The summands with negative constants decay faster than \(|x|^{-N}\), by \eqref{eq:key} and Proposition~\ref{prop:large}, and are absorbed by choosing \(\lambda\) large.

	\begin{theorem}[Existence]\label{thm:existence}
		Let \(N\ge2\), \(n\ge1\), \(p_i\in(1,\infty)\) and \(s_i\in(0,1]\) for \(i\in\{1,\dots,n\}\), with \(\max_is_ip_i<N\), and let \(q>\qL=N/\kb\). Choose an index \(k\) such that \(N/\kappa_k<q\) (for instance, \(\kappa_k=\kb\)). Then, for
		\begin{itemize}
			\item every barrier exponent \(\kappa\in\left[\dfrac{s_kp_k}{q-p_k+1},\ \kappa_k\right)\), and
			\item every amplitude \(A\in\left(0,\ \left(\tfrac12\cc_{s_k,p_k}(-\kappa)\right)^{1/(q-p_k+1)}\right]\),
		\end{itemize}
		there is \(\lambda_0\ge1\), depending only on \(N\), \(k\), \(\kappa\), \(A\) and \(\{(p_i,s_i)\}_{i=1}^n\), such that for every \(\lambda\ge\lambda_0\) the function
		\[
			W(x)\coloneqq A(\lambda+|x|)^{-\kappa}
		\]
		is a positive weak super-solution of \(\Lc u=u^q\) in \(\RN\).
	\end{theorem}

	\begin{proof}
		Fix \(q>\qL\), an index \(k\) with \(N/\kappa_k<q\), and \(\kappa\), \(A\) as in the statement. The interval for \(\kappa\) is non-empty. Indeed, \(q>N/\kappa_k=Q_k(\kappa_k)>p_k-1\), so \(q-p_k+1>0\), and for every \(\tau>0\)
		\begin{equation}\label{eq:equiv}
			\frac{s_kp_k}{q-p_k+1}<\tau\iff p_k-1+\frac{s_kp_k}\tau<q\iff Q_k(\tau)<q,
		\end{equation}
		which holds for \(\tau=\kappa_k\). Let
		\[
			\Nm\coloneqq\{i\in\{1,\dots,n\}:\ \kappa_i<\kappa\},\qquad a_i\coloneqq\kappa(p_i-1)+s_ip_i\quad(1\le i\le n).
		\]
		For \(i\notin\Nm\) we have \(\kappa\le\kappa_i<\frac N{p_i-1}\), so Lemma~\ref{lem:fundamental} and Proposition~\ref{prop:pointwise} apply, and
		\[
			\cc_{s_k,p_k}(-\kappa)>0,\qquad\cc_{s_i,p_i}(-\kappa)\ge0\ \text{ for }i\notin\Nm,
		\]
		where the first inequality holds because \(\kappa<\kappa_k\).

		\medskip

		\noindent{\bf \emph{Step 1}:} Pointwise bound for each operator. Since \(W(x)=A\lambda^{-\kappa}w_\kappa(x/\lambda)\), the scaling identity \eqref{eq:scaling} gives, for every \(i\) and every \(x\ne0\),
		\begin{equation}\label{eq:Wscal}
			(-\Delta_{p_i})^{s_i}W(x)=A^{p_i-1}\lambda^{-a_i}\left[(-\Delta_{p_i})^{s_i}w_\kappa\right](\tfrac x\lambda).
		\end{equation}
		For \(i\notin\Nm\), Proposition~\ref{prop:pointwise} gives
		\[
			(-\Delta_{p_i})^{s_i}W(x)\ge A^{p_i-1}\cc_{s_i,p_i}(-\kappa)\,(\lambda+|x|)^{-a_i}\ge0,
		\]
		with a positive right-hand side for \(i=k\). For \(i\in\Nm\) we set \(b_i\coloneqq\min\{a_i,\,N+s_ip_i\}\). If \(s_i=1\) or \(\kappa(p_i-1)<N\), Proposition~\ref{prop:pointwise} applies, and its constant is negative. If \(s_i\in(0,1)\) and \(\kappa(p_i-1)\ge N\), Proposition~\ref{prop:large} applies. In both cases there is \(C_i>0\) such that
		\[
			\left[(-\Delta_{p_i})^{s_i}w_\kappa\right](y)\ge-C_i(1+|y|)^{-b_i}\log(e+|y|)\qquad\text{for }y\ne0.
		\]
		Since \(b_i\le a_i\) and \(\lambda\ge1\), \eqref{eq:Wscal} gives
		\[
			(-\Delta_{p_i})^{s_i}W(x)\ge-C_iA^{p_i-1}\lambda^{b_i-a_i}(\lambda+|x|)^{-b_i}\log(e+\tfrac{|x|}\lambda)\ge-C_iA^{p_i-1}(\lambda+|x|)^{-b_i}\log(e+\lambda+|x|).
		\]

		\medskip

		\noindent{\bf \emph{Step 2}:} Comparison of the decay rates. By \eqref{eq:key},
		\[
			a_k<N<b_i\qquad\text{for every }i\in\Nm,
		\]
		because \(a_i>N\) for \(i\in\Nm\) and \(N+s_ip_i>N\). We set
		\[
			\gamma\coloneqq\tfrac12\min_{i\in\Nm}(b_i-a_k)>0,
		\]
		with \(\gamma\coloneqq1\) if \(\Nm=\emptyset\). There is \(C_\gamma>0\) such that \(\log(e+t)\le C_\gamma t^\gamma\) for every \(t\ge1\).

		\medskip

		\noindent{\bf \emph{Step 3}:} Absorption. We drop the non-negative terms with \(i\notin\Nm\cup\{k\}\). For \(i\in\Nm\), since \(t\coloneqq\lambda+|x|\ge\lambda\ge1\),
		\[
			t^{-b_i}\log(e+t)\le C_\gamma\,t^{-b_i+\gamma}=C_\gamma\,t^{-a_k}\,t^{-(b_i-a_k-\gamma)}\le C_\gamma\,\lambda^{-\gamma}\,t^{-a_k}.
		\]
		Therefore, for every \(x\ne0\),
		\[
			\Lc W(x)\ge(\lambda+|x|)^{-a_k}\left(A^{p_k-1}\cc_{s_k,p_k}(-\kappa)-\lambda^{-\gamma}C_\gamma\sum_{i\in\Nm}A^{p_i-1}C_i\right).
		\]
		We choose
		\[
			\lambda_0\coloneqq\max\left\{1,\ \left(\frac{2C_\gamma\sum_{i\in\Nm}A^{p_i-1}C_i}{A^{p_k-1}\cc_{s_k,p_k}(-\kappa)}\right)^{1/\gamma}\right\}.
		\]
		Then, for \(\lambda\ge\lambda_0\),
		\begin{equation}\label{eq:absorb}
			\Lc W(x)\ge\tfrac12A^{p_k-1}\cc_{s_k,p_k}(-\kappa)\,(\lambda+|x|)^{-a_k}\qquad\text{for every }x\ne0.
		\end{equation}

		\medskip

		\noindent{\bf \emph{Step 4}:} Comparison with \(W^q\). By \eqref{eq:equiv}, the condition \(\kappa\ge\frac{s_kp_k}{q-p_k+1}\) is equivalent to \(a_k\le\kappa q\). Since \(\lambda+|x|\ge1\), this gives \((\lambda+|x|)^{-a_k}\ge(\lambda+|x|)^{-\kappa q}\). The condition on \(A\) means \(A^{q-p_k+1}\le\frac12\cc_{s_k,p_k}(-\kappa)\). Hence, by \eqref{eq:absorb},
		\[
			\Lc W(x)\ge\tfrac12A^{p_k-1}\cc_{s_k,p_k}(-\kappa)\,(\lambda+|x|)^{-\kappa q}\ge A^q(\lambda+|x|)^{-\kappa q}=W(x)^q\qquad\text{for every }x\in\RN\setminus\{0\}.
		\]

		\medskip

		\noindent{\bf \emph{Step 5}:} Weak formulation. The function \(W\) is positive, bounded and Lipschitz, so \(W\in\DL(\RN)\) and \(W^q\in L^1_\loc(\RN)\). By Lemma~\ref{lem:barrier}, the pointwise inequality of Step~4 implies that \(\Lc W\ge W^q\) weakly in \(\RN\).
	\end{proof}

	\begin{remark}[One-parameter barriers]\label{rem:oneparam}
		If \(\kappa\le\ku\), then \(\Nm=\emptyset\), no absorption is needed, and one may take \(\lambda=1\) and any amplitude with \(A^{q-p_k+1}\le\cc_{s_k,p_k}(-\kappa)\). For \(\kappa=\ku\) the constants of the operators with \(\kappa_i=\ku\) vanish, so one needs an index \(k\) with \(\kappa_k>\ku\), which exists when \(\ku<\kb\). In this way the barriers \(A(1+|x|)^{-\kappa}\) with \(\kappa\le\ku\) cover the range \(q\ge\qz\) if \(\ku<\kb\), and \(q>\qz\) if \(\ku=\kb\), but not more. Since each \(Q_i\) is strictly decreasing, Lemma~\ref{lem:thresholds} gives
		\[
			\qz=\min_{1\le i\le n}Q_i(\ku)>\min_{1\le i\le n}Q_i(\kb)=\frac N\kb\qquad\text{whenever }\ku<\kb.
		\]
		Thus, unless all the decay exponents coincide, the second parameter \(\lambda\) is what allows us to reach the threshold \(N/\kb\).
	\end{remark}

	For \(n=1\) we have \(\Nm=\emptyset\) for every \(\kappa<\kappa_1\), so one may take \(\lambda=1\) and any \(A\) with \(A^{q-p_1+1}\le\cc_{s_1,p_1}(-\kappa)\). In this way Theorem~\ref{thm:existence} recovers \cite[Theorem~6.2]{DPQ26} when \(s_1\in(0,1)\), and its local counterpart \cite{BD02} when \(s_1=1\).

%=====================================================================
\section{The Liouville-type theorem}\label{sec:liouville}

	In this section we prove part (i) of Theorem~\ref{thm:main}, stated below as Theorem~\ref{thm:liouville}: in the range \(0<q<\qL\), every continuous non-negative weak super-solution of \eqref{eq:LE} vanishes identically. By Theorem~\ref{thm:existence}, the threshold \(\qL\) is then sharp. Since \(\E_\Lc=\sum_i\E_{p_i,s_i}\) is linear with respect to the test function, the barrier estimates for the sum follow by adding the single-operator versions term by term. For a fractional operator these estimates are proved in Section~\ref{subsec:trunc}, also in the range of exponents not covered by \cite{DPQ26}. The local order \(s_i=1\) requires its own barrier, which we provide in Section~\ref{subsec:localbarrier}. Throughout this section we rely on the comparison principle and the strong minimum principle of Section~\ref{sec:prelim}.

\subsection{Truncated powers}\label{subsec:trunc}
	A super-solution is bounded from below by comparison with truncated powers \(|x|^\beta\). The only restriction on the exponent is \(\beta<-\kb\), and in the proof of Theorem~\ref{thm:liouville} we take \(\beta\) close to \(-\kb\). Since \(-\kb\le-\kappa_i\) for every \(i\), the fundamental constant of each operator at such a \(\beta\) is negative whenever it is defined (Lemma~\ref{lem:fundamental}).

	Throughout the rest of this section, \(N\ge2\), \(0<\varepsilon<1<r<R\), we write \(B_\varrho\coloneqq B_\varrho(0)\), and
	\begin{equation}\label{eq:annuli}
		A_{r,R}\coloneqq\{x\in\RN:\ r<|x|<R\},\qquad A_r\coloneqq A_{r/2,2r}=\{x\in\RN:\ \tfrac r2<|x|<2r\}.
	\end{equation}
	Given \(\beta<0\), we set the inner truncation radius
	\[
		r_\varepsilon\coloneqq r\left(\frac\varepsilon{1+\varepsilon2^\beta}\right)^{-1/\beta}
	\]
	of \cite{DPQ26}. The ratio \(r_\varepsilon/r\) does not depend on \(r\) and tends to \(0^+\) as \(\varepsilon\to0^+\). We define the two truncated power barriers
	\begin{equation}\label{eq:barriers}
		\varphi_\varepsilon(x)\coloneqq
		\begin{cases}
			\varepsilon^\beta,&\text{if }0\le|x|<\varepsilon,\\
			|x|^\beta,&\text{if }\varepsilon\le|x|,
		\end{cases}
		\qquad
		\psi_\varepsilon(x)\coloneqq
		\begin{cases}
			r_\varepsilon^\beta,&\text{if }0\le|x|\le r_\varepsilon,\\
			|x|^\beta,&\text{if }r_\varepsilon<|x|\le2r,\\
			(2r)^\beta,&\text{if }2r<|x|.
		\end{cases}
	\end{equation}
	Both are radial, bounded and Lipschitz, hence they belong to \(\DL(\RN)\). The function \(\varphi_\varepsilon\) flattens the power \(|x|^\beta\) inside \(B_\varepsilon\), and \(\psi_\varepsilon\) flattens it inside \(B_{r_\varepsilon}\) and outside \(B_{2r}\).

	The following lemma replaces the single-operator estimates of \cite[Lemmas~4.1 and~4.2]{DPQ26} and extends them to every \(\beta<-\kappa\), including the range \(\beta\le-\frac N{p-1}\), where \(|x|^{\beta(p-1)}\) is not locally integrable at the origin. In this lemma \(\varepsilon>0\) and \(r>0\) are arbitrary, that is, we do not use the restriction \(0<\varepsilon<1<r\).

	\begin{lemma}[Truncated powers]\label{lem:trunc}
		Let \(N\ge2\), \(p\in(1,\infty)\) and \(s\in(0,1)\) with \(sp<N\), set \(\kappa\coloneqq\frac{N-sp}{p-1}\), and let \(\beta<-\kappa\). There is \(\theta_0=\theta_0(N,s,p,\beta)\in(0,\frac12]\) such that, for every \(\varepsilon>0\),
		\begin{equation}\label{eq:trunc-sub}
			(-\Delta_p)^s\varphi_\varepsilon(x)\le0\qquad\text{for every }x\text{ with }|x|\ge\varepsilon/\theta_0.
		\end{equation}
		Consequently:
		\begin{enumerate}[label=\textup{(\alph*)}]
			\item if \(1<r<R\) and \(\varepsilon\in(0,\theta_0]\), then \(\E_{p,s}(\varphi_\varepsilon,\varphi)\le0\) for every non-negative \(\varphi\in C^\infty_c(A_{r,R})\);
			\item if \(r>0\) and \(r_\varepsilon/r\le\theta_0/2\), then \((-\Delta_p)^s\psi_\varepsilon\le0\) pointwise in \(A_r\), and \(\E_{p,s}(\psi_\varepsilon,\varphi)\le0\) for every non-negative \(\varphi\in C^\infty_c(A_r)\).
		\end{enumerate}
	\end{lemma}

	\begin{proof}
		\noindent{\bf \emph{Step 1}:} Proof of \eqref{eq:trunc-sub}. Since \(\beta<0\), we have \(\varphi_\varepsilon(x)=\min\{\varepsilon^\beta,|x|^\beta\}=\varepsilon^\beta\varphi_1(x/\varepsilon)\), and the scaling identity \eqref{eq:scaling} gives
		\[
			(-\Delta_p)^s\varphi_\varepsilon(x)=\varepsilon^{\beta(p-1)-sp}\left[(-\Delta_p)^s\varphi_1\right](x/\varepsilon).
		\]
		Hence it suffices to prove \eqref{eq:trunc-sub} for \(\varepsilon=1\), that is, \((-\Delta_p)^s\varphi_1(x)\le0\) for \(|x|\ge1/\theta_0\). The function \(\varphi_1\) is bounded and Lipschitz, and it is smooth with non-vanishing gradient in \(\{|y|>1\}\), so by Lemma~\ref{lem:smooth} the value \((-\Delta_p)^s\varphi_1(x)\) exists for every \(|x|>1\). Fix \(x\) with \(\rho\coloneqq|x|\ge2\). By rotational invariance we may assume \(x=\rho e_1\).

		\medskip

		\noindent{\bf \emph{Case 1}:} \(-\frac N{p-1}<\beta<-\kappa\). By Lemma~\ref{lem:fundamental} and \eqref{eq:csign}, the power \(v_\beta=|\cdot|^\beta\) satisfies \((-\Delta_p)^sv_\beta(x)=\cc\,\rho^{\beta(p-1)-sp}\) with \(\cc\coloneqq\cc_{s,p}(\beta)<0\). The functions \(\varphi_1\) and \(v_\beta\) coincide outside \(B_1\), and \(\varphi_1=1\le|y|^\beta\) in \(B_1\). For \(0<\varepsilon'<\rho-1\) the truncations at \(x\) are taken over the same set, so
		\[
			(-\Delta_p)^s_{\varepsilon'}\varphi_1(x)-(-\Delta_p)^s_{\varepsilon'}v_\beta(x)=2\int_{B_1}\frac{\Psi_p(\rho^\beta-1)-\Psi_p(\rho^\beta-|y|^\beta)}{|x-y|^{N+sp}}\,dy.
		\]
		Here \(\rho^\beta-|y|^\beta\le\rho^\beta-1<0\), so the numerator equals \((|y|^\beta-\rho^\beta)^{p-1}-(1-\rho^\beta)^{p-1}\in[0,|y|^{\beta(p-1)}]\). Moreover, \(|x-y|\ge\rho-1\ge\rho/2\) for \(y\in B_1\). Since \(\beta(p-1)>-N\), letting \(\varepsilon'\to0^+\) we get
		\[
			(-\Delta_p)^s\varphi_1(x)\le\cc\,\rho^{\beta(p-1)-sp}+C_1\rho^{-N-sp}=\rho^{\beta(p-1)-sp}\left(\cc+C_1\rho^{-(N+\beta(p-1))}\right),
		\]
		with \(C_1\coloneqq2^{N+sp+1}\int_{B_1}|y|^{\beta(p-1)}\,dy<\infty\). Since \(N+\beta(p-1)>0\) and \(\cc<0\), the right-hand side is non-positive for \(\rho\ge\rho_1\coloneqq\max\{2,(C_1/|\cc|)^{1/(N+\beta(p-1))}\}\).

		\medskip

		\noindent{\bf \emph{Case 2}:} \(\beta\le-\frac N{p-1}\). We split \(\RN\) into \(E_1\coloneqq\{|y|\ge\rho/2\}\), which contains a neighborhood of \(x\), and \(E_2\coloneqq\{|y|<\rho/2\}\).

		\emph{Contribution of \(E_1\).} Since \(\rho/2\ge1\), we have \(\varphi_1(y)=|y|^\beta\) in \(E_1\). The change of variables \(y=\rho z\) gives, for \(0<\varepsilon'<\rho/2\),
		\[
			2\int_{E_1\setminus B_{\varepsilon'}(x)}\frac{\Psi_p(\varphi_1(x)-\varphi_1(y))}{|x-y|^{N+sp}}\,dy=\rho^{\beta(p-1)-sp}\,2\int_{\{|z|\ge\frac12\}\setminus B_{\varepsilon'/\rho}(e_1)}\frac{\Psi_p(1-|z|^\beta)}{|e_1-z|^{N+sp}}\,dz.
		\]
		The integrand on the right is bounded on \(\{|z|\ge\frac12\}\), since \(|z|^\beta\le2^{-\beta}\) there, and the kernel is integrable away from \(e_1\). Near \(e_1\) the function \(|z|^\beta\) is smooth with gradient \(\beta e_1\ne0\), so by Lemma~\ref{lem:local} the integral on the right converges, as \(\varepsilon'\to0^+\), to a finite number \(K_\beta\) depending only on \(N\), \(s\), \(p\) and \(\beta\). Its sign plays no role, and only \(K_\beta^+\) is used below.

		\emph{Contribution of \(E_2\).} For \(|y|<\rho/2\) we have \(\varphi_1(y)\ge(\rho/2)^\beta=2^{-\beta}\rho^\beta\), hence \(\varphi_1(y)-\rho^\beta\ge(1-2^\beta)\varphi_1(y)>0\), and \(|x-y|\le3\rho/2\). Therefore
		\[
			2\int_{E_2}\frac{\Psi_p(\varphi_1(x)-\varphi_1(y))}{|x-y|^{N+sp}}\,dy\le-c_2\,\rho^{-N-sp}\int_{\{|y|<\rho/2\}}\varphi_1(y)^{p-1}\,dy,\qquad c_2\coloneqq2(1-2^\beta)^{p-1}\left(\tfrac23\right)^{N+sp}.
		\]
		If \(\beta(p-1)<-N\), the last integral is at least \(|B_1|\). If \(\beta(p-1)=-N\), it is at least
		\[
			\int_{\{1<|y|<\rho/2\}}|y|^{-N}\,dy=|S^{N-1}|\log(\rho/2).
		\]
		The contribution of \(E_2\) does not depend on \(\varepsilon'\), so letting \(\varepsilon'\to0^+\) we obtain
		\[
			(-\Delta_p)^s\varphi_1(x)\le
			\begin{cases}
				\rho^{\beta(p-1)-sp}\left(K_\beta^+-c_2|B_1|\,\rho^{-(N+\beta(p-1))}\right),&\text{if }\beta(p-1)<-N,\\
				\rho^{-N-sp}\left(K_\beta^+-c_2|S^{N-1}|\log(\rho/2)\right),&\text{if }\beta(p-1)=-N.
			\end{cases}
		\]
		In the first case \(-(N+\beta(p-1))>0\). In both cases the right-hand side is non-positive for \(\rho\ge\rho_2\), with \(\rho_2\ge2\) depending only on \(N\), \(s\), \(p\) and \(\beta\).

		\medskip

		Taking \(\theta_0\coloneqq1/\rho_1\) in Case~1 and \(\theta_0\coloneqq1/\rho_2\) in Case~2 proves \eqref{eq:trunc-sub}.

		\medskip

		\noindent{\bf \emph{Step 2}:} Proof of \textup{(a)}. If \(\varepsilon\le\theta_0\) and \(x\in A_{r,R}\), then \(|x|>r>1\ge\varepsilon/\theta_0\), so \((-\Delta_p)^s\varphi_\varepsilon\le0\) in \(A_{r,R}\). There \(\varphi_\varepsilon=|x|^\beta\) is smooth with non-vanishing gradient, and \(\varphi_\varepsilon\) is bounded and Lipschitz in \(\RN\). Lemma~\ref{lem:smooth} with \(U=A_{r,R}\) gives the claim.

		\medskip

		\noindent{\bf \emph{Step 3}:} Proof of \textup{(b)}. We compare \(\psi_\varepsilon\) with \(\varphi_{r_\varepsilon}\). These two functions coincide in \(\overline{B_{2r}}\), and \(\psi_\varepsilon(y)=(2r)^\beta\ge|y|^\beta=\varphi_{r_\varepsilon}(y)\) for \(|y|>2r\). Hence, for \(x\in A_r\) and every \(y\), \(\psi_\varepsilon(x)-\psi_\varepsilon(y)\le\varphi_{r_\varepsilon}(x)-\varphi_{r_\varepsilon}(y)\). Since \(\Psi_p\) is increasing, \((-\Delta_p)^s_{\varepsilon'}\psi_\varepsilon(x)\le(-\Delta_p)^s_{\varepsilon'}\varphi_{r_\varepsilon}(x)\). Both functions are smooth with non-vanishing gradient near \(x\), because \(r_\varepsilon<r/2<|x|<2r\). Letting \(\varepsilon'\to0^+\) (Lemma~\ref{lem:smooth}) and using \eqref{eq:trunc-sub} with \(r_\varepsilon\) in place of \(\varepsilon\), which applies since \(|x|>r/2\ge r_\varepsilon/\theta_0\), we obtain \((-\Delta_p)^s\psi_\varepsilon(x)\le0\). The weak inequality follows from Lemma~\ref{lem:smooth} with \(U=A_r\), where \(\psi_\varepsilon=|x|^\beta\).
	\end{proof}

	\begin{remark}\label{rem:trunc}
		In Case~1 of Lemma~\ref{lem:trunc} we compare \(\varphi_1\) with the untruncated power, as in \cite{DPQ26}. The flattening raises the operator, but only by \(O(\rho^{-N-sp})\), and this is beaten by the negative fundamental constant. In Case~2 there is no fundamental constant. The flattening, however, puts a mass of size at least \(|B_1|\) (or \(\log\rho\)) near the origin. This mass gives a negative contribution of order \(\rho^{-N-sp}\), which dominates the term \(K_\beta\rho^{\beta(p-1)-sp}\) coming from the far region, because \(\beta(p-1)\le-N\). This is the same mechanism as in Proposition~\ref{prop:large}. It allows us to treat sums in which \(\kb\ge\frac N{p_i-1}\) for some fractional index \(i\), with no condition relating the different operators.
	\end{remark}

\subsection{A barrier for the local order}\label{subsec:localbarrier}
	Lemma~\ref{lem:trunc} concerns a single \emph{fractional} operator, \(s\in(0,1)\). Before deriving the joint version, we supply the missing element for the order \(s=1\). In this case the argument is shorter than in the nonlocal case, since the local operator only sees a neighborhood of each point. For a single local operator the corresponding results are classical. The power solutions go back to Kichenassamy and V\'eron \cite{KV86}, and the nonexistence for the local Lane--Emden problem to \cite{BD02,MP99}. We include the computation in the form needed for the joint estimates.

	\begin{lemma}[Local barrier, \(s=1\)]\label{lem:localbarrier}
		Let \(p\in(1,\infty)\) with \(p<N\), and let \(\beta<\frac{p-N}{p-1}\). Then the truncations \eqref{eq:barriers} satisfy
		\[
			\E_{p,1}(\varphi_\varepsilon,\varphi)\le0\quad\text{for every non-negative }\varphi\in C^\infty_c(A_{r,R}),
		\]
		and, if \(r_\varepsilon<r/2\),
		\[
			\E_{p,1}(\psi_\varepsilon,\varphi)\le0\quad\text{for every non-negative }\varphi\in C^\infty_c(A_r).
		\]
	\end{lemma}

	\begin{proof}
		Since \(p<N\), we have \(\beta<\frac{p-N}{p-1}<0\). By Lemma~\ref{lem:fundamental},
		\[
			-\Delta_p|x|^\beta=\cc_{1,p}(\beta)\,|x|^{\beta(p-1)-p}\quad\text{in }\RN\setminus\{0\},\qquad\text{with}\quad\cc_{1,p}(\beta)=-\Psi_p(\beta)((p-1)\beta+N-p)<0.
		\]
		Let \(\varphi\in C^\infty_c(A_{r,R})\) be non-negative. Since \(\varepsilon<1<r\), we have \(\varphi_\varepsilon=|x|^\beta\) in \(A_{r,R}\), where this function is smooth with non-vanishing gradient. Integrating by parts,
		\[
			\E_{p,1}(\varphi_\varepsilon,\varphi)=\int_{A_{r,R}}|\nabla\varphi_\varepsilon|^{p-2}\nabla\varphi_\varepsilon\cdot\nabla\varphi\,dx=\cc_{1,p}(\beta)\int_{A_{r,R}}|x|^{\beta(p-1)-p}\varphi\,dx\ \le\ 0.
		\]
		For \(\psi_\varepsilon\) the argument is identical in \(A_r\), because the two flattening radii \(r_\varepsilon<r/2\) and \(2r\) lie outside the open annulus \(\{\frac r2<|x|<2r\}\), so that \(\psi_\varepsilon=|x|^\beta\) there.
	\end{proof}

	\begin{remark}\label{rem:localbarrier}
		Unlike its nonlocal counterpart, Lemma~\ref{lem:localbarrier} does not require \(\varepsilon\) to be small, apart from the condition \(r_\varepsilon<r/2\) needed for \(\psi_\varepsilon\), because the truncation takes place inside the inner ball, which the local operator does not see. For an index \(i\in L\), the condition of Lemma~\ref{lem:localbarrier} is \(\beta<\frac{p_i-N}{p_i-1}=-\kappa_i\), which holds for every \(\beta<-\kb\).
	\end{remark}

\subsection{Hadamard properties}
	Throughout this subsection we fix an exponent \(\beta<-\kb\). The freedom in the choice of \(\beta\) will be used in the proof of Theorem~\ref{thm:liouville}. We use the annuli of \eqref{eq:annuli} and the truncations \(\varphi_\varepsilon\), \(\psi_\varepsilon\) of \eqref{eq:barriers}. For \(i\in F\) let \(\theta_i\coloneqq\theta_0(N,s_i,p_i,\beta)\) be given by Lemma~\ref{lem:trunc}, which applies because \(\beta<-\kappa_i\), and set \(\theta_*\coloneqq\min_{i\in F}\theta_i\) if \(F\ne\emptyset\), and \(\theta_*\coloneqq\frac12\) if \(F=\emptyset\).

	\begin{lemma}[Joint sub-solution of type \(\varphi_\varepsilon\)]\label{lem:joint1}
		There is \(\varepsilon_0\in(0,1)\), independent of \(r\) and \(R\), such that for every \(\varepsilon\in(0,\varepsilon_0)\),
		\[
			\E_{p_i,s_i}(\varphi_\varepsilon,\varphi)\le0\qquad\text{for every }i\in\{1,\dots,n\}\text{ and every non-negative }\varphi\in C^\infty_c(A_{r,R}).
		\]
		In particular, \(\varphi_\varepsilon\) is a weak sub-solution of \(\Lc v=0\) in \(A_{r,R}\).
	\end{lemma}

	\begin{proof}
		Let \(\varepsilon_0\coloneqq\theta_*\). For a fractional index \(i\in F\), Lemma~\ref{lem:trunc}\textup{(a)}, applied to \((p_i,s_i)\), gives the inequality for every \(\varepsilon\in(0,\varepsilon_0)\), since \(\varepsilon_0\le\theta_i\). This does not depend on \(r\) or \(R\). For a local index \(i\in L\) we have \(\beta<-\kb\le-\kappa_i\), so Lemma~\ref{lem:localbarrier} with \(p=p_i\) gives the same inequality for every \(\varepsilon\in(0,1)\). The last assertion follows by adding over \(i\).
	\end{proof}

	\begin{lemma}[Joint sub-solution of type \(\psi_\varepsilon\)]\label{lem:joint2}
		There is \(\varepsilon_0\in(0,1)\), independent of \(r\), such that for every \(\varepsilon\in(0,\varepsilon_0)\) we have \(r_\varepsilon<r/2\) and
		\[
			\E_{p_i,s_i}(\psi_\varepsilon,\varphi)\le0\qquad\text{for every }i\in\{1,\dots,n\}\text{ and every non-negative }\varphi\in C^\infty_c(A_r).
		\]
		In particular, \(\psi_\varepsilon\) is a weak sub-solution of \(\Lc v=0\) in \(A_r\).
	\end{lemma}

	\begin{proof}
		Since \(r_\varepsilon/r\) does not depend on \(r\) and tends to \(0\) as \(\varepsilon\to0^+\), there is \(\varepsilon_0\in(0,1)\), independent of \(r\), such that \(r_\varepsilon/r\le\theta_*/2\le\frac14\) for every \(\varepsilon\in(0,\varepsilon_0)\). For a fractional index \(i\in F\), Lemma~\ref{lem:trunc}\textup{(b)} gives \(\E_{p_i,s_i}(\psi_\varepsilon,\varphi)\le0\). For a local index \(i\in L\), Lemma~\ref{lem:localbarrier} with \(p=p_i\) gives \(\E_{p_i,1}(\psi_\varepsilon,\varphi)\le0\), since \(r_\varepsilon<r/2\).
	\end{proof}

	Observe that Lemmas~\ref{lem:joint1} and~\ref{lem:joint2} give the sign of each summand \(\E_{p_i,s_i}\), and not only of their sum. By Remark~\ref{rem:affine}, every function of the form \(\alpha\varphi_\varepsilon+c\) or \(\alpha\psi_\varepsilon+c\) with \(\alpha>0\) and \(c\in\R\) is then a weak sub-solution of \(\Lc v=0\) in the corresponding annulus. We now prove the two Hadamard-type estimates, following \cite[Section~4]{DPQ26}. For a continuous function \(u\) we write
	\[
		m(r)\coloneqq\min_{\overline{B_r}}u,\qquad r>0.
	\]

	\begin{proposition}[Hadamard-type lower bound]\label{prop:had1}
		Let \(u\in\DL(\RN)\cap C(\RN)\) be a non-negative weak super-solution of \(\Lc u=0\) in \(\RN\). For every \(\beta<-\kb\) and \(r_0>1\) there is a constant \(C>0\), depending only on \(N\), \(\beta\) and \(\{(p_i,s_i)\}_{i=1}^n\), such that
		\[
			m(r)\ge C\,m(r_0)\,r^\beta\qquad\text{for every }r>r_0.
		\]
	\end{proposition}

	\begin{proof}
		If \(m(r_0)=0\) there is nothing to prove, so we assume \(m(r_0)>0\). Fix \(\beta<-\kb\) and \(r_0>1\). By Lemma~\ref{lem:joint1}, applied with \(r=r_0\), there is \(\varepsilon_0\in(0,1)\), independent of \(r_0\) and \(R\), such that every summand of \(\E_\Lc(\varphi_\varepsilon,\cdot)\) is non-positive in \(A_{r_0,R}\) for every \(R>r_0\). We fix \(\varepsilon\in(0,\varepsilon_0)\). For \(R>r_0\), since \(\beta<0\) gives \(\varepsilon^\beta>R^\beta>0\), the number \(\alpha\coloneqq m(r_0)/(\varepsilon^\beta-R^\beta)\) is positive, and we set
		\[
			H_{\varepsilon,R}(x)\coloneqq\alpha(\varphi_\varepsilon(x)-R^\beta)=
			\begin{cases}
				\alpha(|x|^\beta-R^\beta),&\text{if }|x|\ge\varepsilon,\\
				m(r_0),&\text{if }|x|<\varepsilon.
			\end{cases}
		\]
		By Remark~\ref{rem:affine}, \(H_{\varepsilon,R}\) is a weak sub-solution of \(\Lc v=0\) in \(A_{r_0,R}\), so \(\Lc u\ge\Lc H_{\varepsilon,R}\) weakly there. Moreover, \(u\ge H_{\varepsilon,R}\) in \(\RN\setminus A_{r_0,R}\). Indeed, in \(\{|x|\le r_0\}\) we have \(H_{\varepsilon,R}\le m(r_0)\le u\), because \(|x|^\beta\le\varepsilon^\beta\) for \(|x|\ge\varepsilon\), while in \(\{|x|\ge R\}\) we have \(H_{\varepsilon,R}\le0\le u\), because \(|x|^\beta\le R^\beta\). The function \(H_{\varepsilon,R}\) is bounded and Lipschitz, so it belongs to \(\DL(\RN)\cap C(\RN)\). By the comparison principle (Proposition~\ref{prop:comparison}), \(u\ge H_{\varepsilon,R}\) in \(A_{r_0,R}\), that is,
		\[
			u(x)\ \ge\ \frac{m(r_0)}{\varepsilon^\beta-R^\beta}(|x|^\beta-R^\beta),\qquad r_0<|x|<R.
		\]
		Letting \(R\to\infty\), so that \(R^\beta\to0\), we get
		\begin{equation}\label{eq:had1}
			u(x)\ \ge\ \varepsilon^{-\beta}\,m(r_0)\,|x|^\beta\qquad\text{for }|x|>r_0.
		\end{equation}
		Let \(C\coloneqq\varepsilon^{-\beta}>0\), which does not depend on \(r_0\), and let \(r>r_0\). If \(r_0<|x|\le r\), then \(|x|^\beta\ge r^\beta\), and \eqref{eq:had1} yields \(u(x)\ge C\,m(r_0)\,r^\beta\). If \(|x|\le r_0\), then
		\[
			u(x)\ge m(r_0)\ge C\,m(r_0)\,r^\beta,
		\]
		since \(Cr^\beta=(r/\varepsilon)^\beta<1\), as \(\varepsilon<1<r\) and \(\beta<0\). Taking the minimum over \(\overline{B_r}\) we obtain \(m(r)\ge C\,m(r_0)\,r^\beta\).
	\end{proof}

	\begin{proposition}[Doubling estimate]\label{prop:had2}
		Let \(u\in\DL(\RN)\cap C(\RN)\) be a non-negative weak super-solution of \(\Lc u=0\) in \(\RN\). For every \(\beta<-\kb\) there is a constant \(C>1\), depending only on \(N\), \(\beta\) and \(\{(p_i,s_i)\}_{i=1}^n\), such that
		\[
			m\left(\frac r2\right)\le C\,m(r)\qquad\text{for every }r>1.
		\]
	\end{proposition}

	\begin{proof}
		Fix \(\beta<-\kb\) and \(r>1\), and assume \(m(r/2)>0\), since otherwise there is nothing to prove. By Lemma~\ref{lem:joint2} there is \(\varepsilon_0\in(0,1)\), independent of \(r\), such that for \(\varepsilon\in(0,\varepsilon_0)\) we have \(r_\varepsilon<r/2\) and every summand of \(\E_\Lc(\psi_\varepsilon,\cdot)\) is non-positive in \(A_r\). We fix \(\varepsilon\in(0,\varepsilon_0)\). Since \(\beta<0\) gives \(r_\varepsilon^\beta>(2r)^\beta>0\), we may define
		\[
			\widetilde H_{\varepsilon,r}(x)\coloneqq m\left(\frac r2\right)\frac{\psi_\varepsilon(x)-(2r)^\beta}{r_\varepsilon^\beta-(2r)^\beta}.
		\]
		By Remark~\ref{rem:affine}, \(\widetilde H_{\varepsilon,r}\) is a weak sub-solution of \(\Lc v=0\) in \(A_r\). Moreover, \(u\ge\widetilde H_{\varepsilon,r}\) in \(\RN\setminus A_r\). Indeed, in \(\{|x|\le r/2\}\) we have \(\psi_\varepsilon\le r_\varepsilon^\beta\), so \(\widetilde H_{\varepsilon,r}\le m(r/2)\le u\), while in \(\{|x|\ge2r\}\) we have \(\psi_\varepsilon=(2r)^\beta\), so \(\widetilde H_{\varepsilon,r}=0\le u\). By the comparison principle (Proposition~\ref{prop:comparison}), \(u\ge\widetilde H_{\varepsilon,r}\) in \(\RN\). Because \(\beta<0\), the profile \(\psi_\varepsilon\) is non-increasing in \(|x|\) and equals \(r^\beta\) on \(\{|x|=r\}\), so \(\widetilde H_{\varepsilon,r}\) attains its minimum over \(\overline{B_r}\) there. Hence
		\[
			m(r)\ \ge\ \min_{\overline{B_r}}\widetilde H_{\varepsilon,r}=m\left(\frac r2\right)\frac{r^\beta-(2r)^\beta}{r_\varepsilon^\beta-(2r)^\beta}=m\left(\frac r2\right)\frac{1-2^\beta}{(r_\varepsilon/r)^\beta-2^\beta}.
		\]
		By the choice of \(r_\varepsilon\) we have \((r_\varepsilon/r)^\beta=(1+\varepsilon2^\beta)/\varepsilon\), and therefore
		\[
			m(r)\ \ge\ \varepsilon(1-2^\beta)\,m\left(\frac r2\right).
		\]
		Since \(\varepsilon\in(0,1)\) and \(2^\beta\in(0,1)\), the constant \(\varepsilon(1-2^\beta)\) lies in \((0,1)\), and it depends only on \(N\), \(\beta\) and \(\{(p_i,s_i)\}_{i=1}^n\), through the choice of \(\varepsilon\). Taking \(C\coloneqq(\varepsilon(1-2^\beta))^{-1}>1\) yields \(m(r/2)\le C\,m(r)\) for every \(r>1\).
	\end{proof}

\subsection{Contact estimate}
	The last ingredient is an upper bound for \(m(R)\) which, in the proof of the Liouville theorem, will be incompatible with the lower bound of Proposition~\ref{prop:had1}. It is the analogue for the sum \(\Lc\) of the contact argument of \cite[Section~6]{DPQ26}. We obtain it directly from the weak comparison principle, using the cut-off \(\mu\) of \eqref{eq:mu} and the constant \(C^*\) of \eqref{eq:Cstar}.

	\begin{proposition}[Contact estimate]\label{prop:contact}
		Let \(q>0\), and let \(u\in\DL(\RN)\cap C(\RN)\) be a positive weak super-solution of \(\Lc u=u^q\) in \(\RN\). Then
		\begin{equation}\label{eq:contact}
			m(R)^q\le C^*\sum_{i=1}^nm\left(\frac R2\right)^{p_i-1}R^{-s_ip_i}\qquad\text{for every }R>0,
		\end{equation}
		where \(m(r)=\min_{\overline{B_r}}u\).
	\end{proposition}

	\begin{proof}
		Fix \(R>0\) and \(\eta>0\), and set
		\[
			\zeta(x)\coloneqq\left(m\left(\frac R2\right)+\eta\right)\mu\!\left(\frac{|x|}R\right),\qquad K_\eta\coloneqq C^*\sum_{i=1}^n\left(m\left(\frac R2\right)+\eta\right)^{p_i-1}R^{-s_ip_i}.
		\]
		By the scaling identity \eqref{eq:scaling} and the definition \eqref{eq:Cstar} of \(C^*\), we have \(\Lc\zeta\le K_\eta\) pointwise in \(\RN\), and by Lemma~\ref{lem:flat},
		\begin{equation}\label{eq:zeta}
			\E_\Lc(\zeta,\varphi)=\int_{\RN}\Lc\zeta\,\varphi\,dx\le K_\eta\int_{\RN}\varphi\,dx\qquad\text{for every non-negative }\varphi\in C^\infty_c(\RN).
		\end{equation}
		We claim that \(m(R)^q<K_\eta\). Letting \(\eta\to0^+\) then gives \eqref{eq:contact}.

		Suppose, to the contrary, that \(m(R)^q\ge K_\eta\). The continuous function \(u-\zeta\) equals \(u>0\) in \(\{|x|\ge R\}\), and at a point \(\bar x\in\overline{B_{R/2}}\) with \(u(\bar x)=m(R/2)\) it equals \(-\eta\). Hence
		\[
			t_*\coloneqq\min_{\RN}(u-\zeta)\le-\eta<0,
		\]
		and the minimum is attained at some point \(x_R\in B_R\). Fix \(\delta\in(0,\eta)\) and consider
		\[
			\Omega\coloneqq\{x\in\RN:\ u(x)-\zeta(x)<t_*+\delta\},\qquad v\coloneqq\zeta+t_*+\delta.
		\]
		The set \(\Omega\) is open and contains \(x_R\). It is contained in \(B_R\), because \(u-\zeta=u>0>t_*+\delta\) in \(\{|x|\ge R\}\). By construction, \(u\ge v\) in \(\RN\setminus\Omega\). Since \(u\ge m(R)\) in \(\Omega\subset B_R\), for every non-negative \(\varphi\in C^\infty_c(\Omega)\) we have, by \eqref{eq:zeta} and Remark~\ref{rem:affine},
		\[
			\E_\Lc(u,\varphi)\ge\int_\Omega u^q\varphi\,dx\ge m(R)^q\int_\Omega\varphi\,dx\ge K_\eta\int_\Omega\varphi\,dx\ge\E_\Lc(\zeta,\varphi)=\E_\Lc(v,\varphi).
		\]
		That is, \(\Lc u\ge\Lc v\) weakly in \(\Omega\). Both \(u\) and \(v\) belong to \(\DL(\RN)\cap C(\RN)\), so Proposition~\ref{prop:comparison} gives \(u\ge v\) in \(\Omega\), that is, \(u-\zeta\ge t_*+\delta\) in \(\Omega\). This contradicts the fact that \(x_R\in\Omega\) and \(u(x_R)-\zeta(x_R)=t_*\).
	\end{proof}

\subsection{Proof of the Liouville-type theorem}
	We now combine the strong minimum principle, the contact estimate and the two Hadamard estimates, following the outline of \cite[Section~6]{DPQ26}. A non-trivial super-solution is bounded from below by a fixed negative power, while the contact estimate forces it to decay faster, and the two facts are incompatible in the range \(q<\qL\).

	\begin{theorem}[Liouville-type theorem]\label{thm:liouville}
		Let \(N\ge2\), \(n\ge1\), \(p_i\in(1,\infty)\) and \(s_i\in(0,1]\) for \(i\in\{1,\dots,n\}\), with \(\max_is_ip_i<N\). If \(0<q<\qL\) and \(u\in\DL(\RN)\cap C(\RN)\) is a non-negative weak super-solution of \eqref{eq:LE} (in the sense of Definition~\ref{def:weak}), then \(u\equiv0\).
	\end{theorem}

	\begin{proof}
		Suppose, by contradiction, that \(u\not\equiv0\). Since \(u^q\ge0\), the function \(u\) is a non-negative weak super-solution of \(\Lc u=0\) in \(\RN\), and the strong minimum principle (Proposition~\ref{prop:smp}) gives \(u>0\) in \(\RN\). Thus \(m(r)=\min_{\overline{B_r}}u>0\) for every \(r>0\).

		Fix any \(\beta_0<-\kb\) and let \(C_H>1\) be the constant of Proposition~\ref{prop:had2} for this exponent. Combining \eqref{eq:contact} with \(m(R/2)\le C_H\,m(R)\) and dividing by \(m(R)^q>0\), we get
		\begin{equation}\label{eq:sumterms}
			1\le C_1\sum_{i=1}^nm(R)^{p_i-1-q}R^{-s_ip_i}\qquad\text{for every }R>1,
		\end{equation}
		with \(C_1\coloneqq C^*C_H^{\max_i(p_i-1)}\). We show that, for a suitable choice of \(\beta\), every summand tends to \(0\) as \(R\to\infty\), which contradicts \eqref{eq:sumterms}.

		We first choose \(\beta\). By hypothesis and by Lemma~\ref{lem:thresholds}, \(q<\qL=\min_iQ_i(\kb)\). For every \(i\) with \(p_i-1<q\) we have \(q-p_i+1>0\), and therefore
		\[
			-\kb+\frac{s_ip_i}{q-p_i+1}>0\iff q<p_i-1+\frac{s_ip_i}\kb=Q_i(\kb),
		\]
		which holds. These are finitely many strict inequalities, so we can fix \(\beta<-\kb\) so close to \(-\kb\) that
		\begin{equation}\label{eq:betachoice}
			\beta+\frac{s_ip_i}{q-p_i+1}>0\qquad\text{for every }i\text{ with }p_i-1<q.
		\end{equation}
		By Proposition~\ref{prop:had1} with a fixed \(r_0>1\), there is \(C_\beta>0\) such that
		\begin{equation}\label{eq:lowerm}
			m(R)\ge C_\beta\,m(r_0)\,R^\beta\qquad\text{for every }R>r_0.
		\end{equation}

		Now fix \(i\in\{1,\dots,n\}\). If \(p_i-1\ge q\), then \(m(R)^{p_i-1-q}\le u(0)^{p_i-1-q}\), since \(0<m(R)\le u(0)\), and the \(i\)-th summand in \eqref{eq:sumterms} is at most \(C_1u(0)^{p_i-1-q}R^{-s_ip_i}\), which tends to \(0\). If \(p_i-1<q\), the exponent \(p_i-1-q\) is negative, and \eqref{eq:lowerm} gives
		\[
			m(R)^{p_i-1-q}R^{-s_ip_i}\le\left(C_\beta\,m(r_0)\right)^{p_i-1-q}R^{-(q-p_i+1)\left(\beta+\frac{s_ip_i}{q-p_i+1}\right)}\qquad\text{for }R>r_0.
		\]
		By \eqref{eq:betachoice} the exponent of \(R\) is negative, so this summand also tends to \(0\). Letting \(R\to\infty\) in \eqref{eq:sumterms} we obtain \(1\le0\). This contradiction shows that \(u\equiv0\).
	\end{proof}

	\begin{remark}[Single-operator case and the endpoint \(q=\qL\)]\label{rem:endpoint}
		For \(n=1\),
		\[
			\qL=p_1-1+\frac{s_1p_1}{\kappa_1}=\frac{(p_1-1)\kappa_1+s_1p_1}{\kappa_1}=\frac{(p_1-1)N}{N-s_1p_1}.
		\]
		Theorem~\ref{thm:liouville} reduces to \cite[Theorem~6.1]{DPQ26} when \(s_1\in(0,1)\), and to the nonexistence result for the \(p\)-Laplacian of \cite{BD02,MP99,SZ02} when \(s_1=1\). For a single local operator the endpoint \(q=\qL\) also belongs to the nonexistence range, see \cite{MP99,SZ02}. For \(s_1\in(0,1)\), the endpoint is studied in the recent preprint \cite{Liu25}, which contains several interesting ideas. We have not been able to follow one step of the argument in the critical case: the lower bound at infinity in \cite[Theorem~3.13]{Liu25} seems to require an additional multiplicative constant, since, for instance, the bounded weak super-solution \(\varepsilon\min\{1,|x|^{-\kappa_1}\}\) with \(\varepsilon\in(0,1)\) does not satisfy it as stated. We may well be missing something, and we believe that the approach of \cite{Liu25} can be completed. In the meantime, to the best of our knowledge, the case \(q=\qL\) is still open for a single fractional \(p\)-Laplacian, as already pointed out in \cite{DPQ26}. For \(-\Delta+(-\Delta)^s\) it was included in \cite{GLX26}.

		For the sum, the proof of Theorem~\ref{thm:liouville} stops at \(q=\qL\), but only because of the operators with \(\kappa_i=\kb\). For \(q=\qL\) and \(\kappa_i<\kb\) we still have \(q<Q_i(\kb)\) by Lemma~\ref{lem:thresholds}, so the corresponding summands in \eqref{eq:sumterms} tend to \(0\) as before. For \(\kappa_i=\kb\), instead, \(s_ip_i/(q-p_i+1)=\kb\), so \eqref{eq:betachoice} fails for every \(\beta<-\kb\), and the lower bound \(m(R)\ge cR^\beta\) only shows that such a summand is bounded by a positive power of \(R\). To conclude one would need a sharper lower bound, such as \(m(R)\ge cR^{-\kb}\log R\) for \(R\) large. Even the weaker bound \(m(R)\ge cR^{-\kb}\) does not follow from our arguments: at \(\beta=-\kb\) the fundamental constant of the operators with \(\kappa_i=\kb\) vanishes, and the flattening error has the wrong sign (Case~1 of Lemma~\ref{lem:trunc}). We do not know how to obtain these bounds, even when the operators with \(\kappa_i=\kb\) are linear, and we leave the endpoint for sums as an open problem.
	\end{remark}

%=====================================================================
\section{Examples}\label{sec:examples}

	We illustrate Theorem~\ref{thm:main} with three examples. They compare the sharp threshold \(\qL\) with the threshold \(\qz\) of the one-parameter barriers of Remark~\ref{rem:oneparam}, and they show the absorption of negative terms in Theorem~\ref{thm:existence}, including the case covered by Proposition~\ref{prop:large}. In (b), the nonexistence part relies on Case~2 of Lemma~\ref{lem:trunc}.

	\begin{example}\label{ex:three}
		In each case we list \(\kappa_i=\frac{N-s_ip_i}{p_i-1}\) and compute \(\qL\) and \(\qz\).
		\begin{enumerate}[label=\textup{(\alph*)}]
			\item \emph{A Laplacian and a half-Laplacian.} On \(\R^3\),
			\[
				\Lc u=-\Delta u+(-\Delta)^{1/2}u,
			\]
			so \(N=3\), \(n=2\), \(p_1=p_2=2\), \(s_1=1\), \(s_2=1/2\). Then \(\kappa_1=1\) and \(\kappa_2=2\), so \(\qL=\frac32\).
			\begin{itemize}
				\item For \(q>3/2\), Theorem~\ref{thm:existence} with \(k=2\) gives super-solutions \(A(\lambda+|x|)^{-\kappa}\) for every \(\kappa\in\left[\frac1{q-1},2\right)\).
				\item When \(\kappa\in(1,2)\) the local term has a negative constant,
				\[
					-\Delta w_\kappa(x)\ge-\kappa(\kappa-1)(1+|x|)^{-\kappa-2},
				\]
				and it decays faster than the positive nonlocal term \(\cc_{1/2,2}(-\kappa)(1+|x|)^{-\kappa-1}\).
				\item For \(q<3/2\), Theorem~\ref{thm:liouville} gives nonexistence.
			\end{itemize}
			The threshold \(3/2=N/(N-2s)\) agrees with \cite{GLX26}. By comparison, the barriers of Remark~\ref{rem:oneparam} with \(\lambda=1\) only cover \(q\ge\qz=\min_iQ_i(1)=2\).

			\item \emph{Three terms, with a large decay exponent.} On \(\R^5\),
			\[
				\Lc u=-\Delta u+(-\Delta_3)^{1/3}u+(-\Delta)^{3/4}u,
			\]
			so \(N=5\), \(n=3\), with
			\[
				\begin{array}
					{c|cccc}
					i&p_i&s_i&s_ip_i&\kappa_i\\
					\hline
					1&2&1&2&3\\
					2&3&1/3&1&2\\
					3&2&3/4&3/2&7/2
				\end{array}
			\]
			With \(k=3\), Theorem~\ref{thm:existence} gives super-solutions for every
			\[
				q>\qL=\frac N\kb=\frac{10}7,
			\]
			with \(\kappa\in\left[\frac3{2(q-1)},\frac72\right)\), while \(\qz=7/4\). For \(\kappa>2\) the second operator has a negative constant. For \(\kappa\ge5/2\) we also have \(\kappa(p_2-1)\ge N\), so its contribution is estimated by Proposition~\ref{prop:large}. For \(\kappa>3\) the local operator also has a negative constant. All these terms are absorbed. On the nonexistence side, \(\kb=7/2>5/2=N/(p_2-1)\), so every \(\beta<-7/2\) satisfies \(\beta(p_2-1)<-N\), and \(|x|^{\beta(p_2-1)}\) is not locally integrable at the origin. Nevertheless, for \(0<q<10/7\), Theorem~\ref{thm:liouville} gives nonexistence, by means of Case~2 of Lemma~\ref{lem:trunc} for the second operator. Thus \(10/7\) is the critical exponent.

			\item \emph{Equal decay exponents.} On \(\R^3\),
			\[
				\Lc u=-\Delta u+(-\Delta_3)^{1/3}u,
			\]
			so \(N=3\), \(n=2\), \(p_1=2\), \(s_1=1\), \(p_2=3\), \(s_2=1/3\). Then \(\kappa_1=\kappa_2=1\) and \(\qL=\qz=3\). In this case no absorption is needed.
		\end{enumerate}
	\end{example}

%=====================================================================
	\section*{Acknowledgments}
	L.\,M.~Del Pezzo was partially supported by ANII (Agencia Nacional de Investigaci\'on e Innovaci\'on, Uruguay) grant FCE\_3\_2024\_1\_181302. L.\,M.~Del Pezzo and A.~Quaas were partially supported by grant PID2023-148028NB-I00, funded by MICIU/\allowbreak AEI/\allowbreak 10.13039/\allowbreak 501100011033 and by ERDF/EU.

	A.~Quaas was partially supported by FONDECYT grant 1231585. A.~Quaas and B.~Pichucho were partially supported by the MATH-AmSud project AMSUD230019. B.~Pichucho was partially supported by ANID-Subdirecci\'on de Capital Humano/\allowbreak Doctorado Nacional/\allowbreak 21220159.

	B.~Pichucho thanks the Facultad de Ciencias Econ\'omicas y de Administraci\'on, Universidad de la Rep\'ublica, for its hospitality during a research stay in which part of this work was carried out.


\begin{thebibliography}{99}

	\bibitem{BDPQ25} B.~Barrios, L.\,M.~Del Pezzo, A.~Quaas, \emph{Mixed local and nonlocal Laplacian without standard critical exponent for Lane--Emden equation}, preprint (2025), arXiv:2507.12258.

	\bibitem{BM21} B.~Barrios, M.~Medina, \emph{Equivalence of weak and viscosity solutions in fractional non-homogeneous problems}, Math.\ Ann.\ \textbf{381} (2021), no.~3--4, 1979--2012.

	\bibitem{BBF26} M.~Bhakta, A.~Biswas, R.~Filippucci, \emph{Liouville properties for differential inequalities with \((p,q)\) Laplacian operator}, J.\ Lond.\ Math.\ Soc.\ (2) \textbf{113} (2026), Paper No.~e70490.

	\bibitem{BDVV22} S.~Biagi, S.~Dipierro, E.~Valdinoci, E.~Vecchi, \emph{Mixed local and nonlocal elliptic operators: regularity and maximum principles}, Comm.\ Partial Differential Equations \textbf{47} (2022), no.~3, 585--629.

	\bibitem{BD02} I.~Birindelli, F.~Demengel, \emph{Some Liouville theorems for the \(p\)-Laplacian}, in: Proceedings of the 2001 Luminy Conference on Quasilinear Elliptic and Parabolic Equations and Systems, Electron.\ J.\ Differ.\ Equ.\ Conf.\ \textbf{8} (2002), 35--46.

	\bibitem{BS22} C.~Bucur, M.~Squassina, \emph{An asymptotic expansion for the fractional \(p\)-Laplacian and for gradient-dependent nonlocal operators}, Commun.\ Contemp.\ Math.\ \textbf{24} (2022), Paper No.~2150021.

	\bibitem{DPQ17} L.\,M.~Del Pezzo, A.~Quaas, \emph{A Hopf's lemma and a strong minimum principle for the fractional \(p\)-Laplacian}, J.\ Differential Equations \textbf{263} (2017), no.~1, 765--778.

	\bibitem{DPQ26} L.\,M.~Del Pezzo, A.~Quaas, \emph{The fundamental solution of the fractional \(p\)-Laplacian}, NoDEA Nonlinear Differential Equations Appl.\ \textbf{33} (2026), Paper No.~62.

	\bibitem{DKP16} A.~Di Castro, T.~Kuusi, G.~Palatucci, \emph{Local behavior of fractional \(p\)-minimizers}, Ann.\ Inst.\ H.\ Poincar\'e C Anal.\ Non Lin\'eaire \textbf{33} (2016), no.~5, 1279--1299.

	\bibitem{DNPV12} E.~Di Nezza, G.~Palatucci, E.~Valdinoci, \emph{Hitchhiker's guide to the fractional Sobolev spaces}, Bull.\ Sci.\ Math.\ \textbf{136} (2012), no.~5, 521--573.

	\bibitem{DPSV25} S.~Dipierro, K.~Perera, C.~Sportelli, E.~Valdinoci, \emph{An existence theory for nonlinear superposition operators of mixed fractional order}, Commun.\ Contemp.\ Math.\ \textbf{27} (2025), no.~8, Paper No.~2550005.

	\bibitem{DPLSV25} S.~Dipierro, E.~Proietti Lippi, C.~Sportelli, E.~Valdinoci, \emph{A general theory for the \((s,p)\)-superposition of nonlinear fractional operators}, Nonlinear Anal.\ Real World Appl.\ \textbf{82} (2025), Paper No.~104251.

	\bibitem{EJN26} H.~El Jeddaoui, D.~Nabab, \emph{New link between the fractional \(p\)-Laplacian operators and a class of McKean--Vlasov flight type processes}, preprint (2026), arXiv:2609.19492.

	\bibitem{FQ11} P.~Felmer, A.~Quaas, \emph{Fundamental solutions and Liouville type theorems for nonlinear integral operators}, Adv.\ Math.\ \textbf{226} (2011), no.~3, 2712--2738.

	\bibitem{GK22} P.~Garain, J.~Kinnunen, \emph{On the regularity theory for mixed local and nonlocal quasilinear elliptic equations}, Trans.\ Amer.\ Math.\ Soc.\ \textbf{375} (2022), no.~8, 5393--5423.

	\bibitem{Gid80} B.~Gidas, \emph{Symmetry properties and isolated singularities of positive solutions of nonlinear elliptic equations}, in: Nonlinear Partial Differential Equations in Engineering and Applied Science, Lecture Notes in Pure and Appl.\ Math., vol.~54, Dekker, New York, 1980, 255--273.

	\bibitem{GS81a} B.~Gidas, J.~Spruck, \emph{A priori bounds for positive solutions of nonlinear elliptic equations}, Comm.\ Partial Differential Equations \textbf{6} (1981), no.~8, 883--901.

	\bibitem{GS81b} B.~Gidas, J.~Spruck, \emph{Global and local behavior of positive solutions of nonlinear elliptic equations}, Comm.\ Pure Appl.\ Math.\ \textbf{34} (1981), no.~4, 525--598.

	\bibitem{GLX26} Y.~Guo, C.~Li, J.~Xie, \emph{Liouville theorems for the Lane--Emden equation involving a mixed local--nonlocal operator}, preprint (2026), arXiv:2606.16940.

	\bibitem{IMS16} A.~Iannizzotto, S.~Mosconi, M.~Squassina, \emph{Global H\"older regularity for the fractional \(p\)-Laplacian}, Rev.\ Mat.\ Iberoam.\ \textbf{32} (2016), no.~4, 1353--1392.

	\bibitem{KV86} S.~Kichenassamy, L.~V\'eron, \emph{Singular solutions of the \(p\)-Laplace equation}, Math.\ Ann.\ \textbf{275} (1986), no.~4, 599--615.

	\bibitem{KL26} T.~Kim, T.~Lee, \emph{Liouville theorems for nonlocal Lane--Emden inequalities with degenerate kernels}, preprint (2026), arXiv:2602.13822.

	\bibitem{KKL19} J.~Korvenp\"a\"a, T.~Kuusi, E.~Lindgren, \emph{Equivalence of solutions to fractional \(p\)-Laplace type equations}, J.\ Math.\ Pures Appl.\ (9) \textbf{132} (2019), 1--26.

	\bibitem{LG26} R.~Lakshmi, S.~Ghosh, \emph{On weak and viscosity solutions to a nonhomogeneous mixed local--nonlocal equation}, preprint (2026), arXiv:2606.20099.

	\bibitem{LWX26} C.~Li, L.~Wang, J.~Xie, \emph{Asymptotic expansions of the fundamental solution for a mixed local--nonlocal operator}, preprint (2026), arXiv:2609.01403.

	\bibitem{Liu25} L.~Liu, \emph{Nontrivial nonnegative weak solutions to fractional \(p\)-Laplace inequalities}, preprint (2025), arXiv:2501.11989.

	\bibitem{Mar91} P.~Marcellini, \emph{Regularity and existence of solutions of elliptic equations with \(p,q\)-growth conditions}, J.\ Differential Equations \textbf{90} (1991), no.~1, 1--30.

	\bibitem{MP99} E.~Mitidieri, S.\,I.~Pohozaev, \emph{Nonexistence of positive solutions for quasilinear elliptic problems on \(\RN\)}, Proc.\ Steklov Inst.\ Math.\ \textbf{227} (1999), 186--216.

	\bibitem{PS07} P.~Pucci, J.~Serrin, \emph{The Maximum Principle}, Progress in Nonlinear Differential Equations and their Applications, vol.~73, Birkh\"auser, Basel, 2007.

	\bibitem{Ser64} J.~Serrin, \emph{Local behavior of solutions of quasi-linear equations}, Acta Math.\ \textbf{111} (1964), 247--302.

	\bibitem{Ser65} J.~Serrin, \emph{Isolated singularities of solutions of quasi-linear equations}, Acta Math.\ \textbf{113} (1965), 219--240.

	\bibitem{SZ02} J.~Serrin, H.~Zou, \emph{Cauchy--Liouville and universal boundedness theorems for quasilinear elliptic equations and inequalities}, Acta Math.\ \textbf{189} (2002), no.~1, 79--142.

	\bibitem{Vaz84} J.\,L.~V\'azquez, \emph{A strong maximum principle for some quasilinear elliptic equations}, Appl.\ Math.\ Optim.\ \textbf{12} (1984), no.~3, 191--202.

\end{thebibliography}
\end{document}